\documentclass[
11pt,  reqno]{amsart}
\usepackage{amsmath,amssymb,amscd,amsthm}

\usepackage{color}
\usepackage{hyperref}

\usepackage{cases}

\usepackage{bm}

\allowdisplaybreaks[3]

\newtheorem{theorem}{Theorem} [section]

\newtheorem{lemma}[theorem]{Lemma}
\newtheorem{proposition}[theorem]{Proposition}
\newtheorem{remark}[theorem]{Remark}
\newtheorem{definition}[theorem]{Definition}

\DeclareMathOperator*{\supp}{supp}

\newcommand{\Z}{\mathbb{Z}}
\newcommand{\R}{\mathbb{R}}
\newcommand{\C}{\mathbb{C}}
\newcommand{\T}{\mathbb{T}}

\newcommand{\F}{\mathcal{F}}

\newcommand{\al}{\alpha}
\newcommand{\be}{\beta}
\newcommand{\dl}{\delta}

\newcommand{\eps}{\varepsilon}

\newcommand{\G}{\Gamma}

\newcommand{\s}{\sigma}

\newcommand{\ft}{\widehat}
\newcommand{\Ft}{{\mathcal{F}}}
\newcommand{\wt}{\widetilde}
\newcommand{\cj}{\overline}
\newcommand{\dx}{\partial_x}

\newcommand{\dt}{\partial_t}

\newcommand{\pd}{\partial}

\newcommand{\ta}{\theta}

\renewcommand{\l}{\ell}

\newcommand{\les}{\lesssim}
\newcommand{\ges}{\gtrsim}

\newcommand{\jb}[1]
{\langle #1 \rangle}

\newcommand{\up}{\nearrow}
\newcommand{\down}{\searrow}

\newcommand{\ind}{\mathbf 1}

\newcommand{\D}{\mathcal{D}}
\newcommand{\Ml}{\mathcal{M}}
\newcommand{\Nl}{\mathcal{N}}

\newcommand{\X}{\mathcal{X}}

\newcommand{\N}{\mathbb{N}}

\newcommand{\CC}{\mathfrak C}

\newcommand{\fm}{\mathfrak m}

\newcommand{\I}{\mathrm{I}}
\newcommand{\II}{\mathrm{I \! I}}
\newcommand{\III}{\mathrm{I \! I \! I}}

\newcommand{\dual}[2]{\langle #1 , #2 \rangle}
\newcommand{\Dual}[2]{\Big \langle #1 , #2 \Big \rangle}

\newcommand{\floor}[1]{\lfloor #1 \rfloor}

\newcommand{\ceil}[1]{\lceil #1 \rceil}

\newcommand{\Ceil}[1]{\Big \lceil #1 \Big \rceil}

\numberwithin{equation}{section}
\numberwithin{theorem}{section}

\let\Re=\undefined\DeclareMathOperator*{\Re}{Re}

\title[Norm inflation for quadratic derivative fNLS]
{Norm inflation for quadratic derivative fractional nonlinear Schr\"odinger equations}

\author[T.~Kondo and M.~Okamoto]
{Toshiki Kondo and Mamoru Okamoto}

\address{
Toshiki Kondo\\
Department of Mathematics\\\
Graduate School of Advanced Science and Engineering\\
Hiroshima University\\
1-3-1 Kagamiyama,
Higashi-Hiroshima, 739-8526
Japan}
\email{d263049@hiroshima-u.ac.jp}

\address{
Mamoru Okamoto\\
Department of Mathematics\\\
Graduate School of Advanced Science and Engineering\\
Hiroshima University\\
1-3-1 Kagamiyama,
Higashi-Hiroshima, 739-8526
Japan}
\email{mokamoto@hiroshima-u.ac.jp}



\begin{document}

\begin{abstract}
We consider the Cauchy problem for quadratic derivative fractional nonlinear Schr\"odinger equations on $\R$ or $\T$.
We determine the sharp exponents of the fractional derivatives for which the Cauchy problem is well-posed in the Sobolev space.
Thanks to the global well-posedness result established by Nakanishi and Wang (2025),
we can expand the solution as a sum of iterated terms.
By deriving estimates for each iterated term,
we establish norm inflation with infinite loss of regularity,
which in particular implies ill-posedness.
\end{abstract}

\maketitle

\tableofcontents

\newpage

\section{Introduction}

We consider the following Cauchy problem for the quadratic derivative fractional nonlinear Schr\"odinger equation (fNLS): 
\begin{equation}
\left\{
\begin{aligned}
&\dt u + i D^\al u = u D^\be u,
\\
&u|_{t=0} = \phi,
\end{aligned}
\right.
\label{dNLS}
\end{equation}
where $\al, \be > 0$, $u : \R \times \Ml \to \C$ is an unknown function, and $\phi : \Ml \to \C$ is a given function.
Here, $D^\al := \Ft^{-1} | \xi |^\al \Ft$
and
$\Ml$ denotes $\R$ or $\T := \R /(2\pi \Z)$.
The goal in this paper is to determine the sharp exponents $\al, \be$ for which \eqref{dNLS} is well-posed
in the Sobolev space $H^s(\Ml)$ for large $s \in \R$.

When the nonlinearity in \eqref{dNLS} has no derivative ($\be=0$):
\begin{equation}
\left\{
\begin{aligned}
&\dt u + i D^\al u = u^2,
\\
&u|_{t=0} = \phi,
\end{aligned}
\right.
\label{qNLS}
\end{equation}
the standard contraction mapping theorem implies that the Cauchy problem is well-posed in $H^s(\Ml)$ for $s>\frac 12$,
since the embedding
$H^s(\Ml) \hookrightarrow L^\infty(\Ml)$ holds.%
\footnote{%
Many researchers have also studied well-posedness in spaces of low regularity.
In particular, sharp well-posedness for \eqref{qNLS} on $\R$ is obtained in \cite{BeTa06,Kis19} for $\al=2$
and \cite{RLY18} for $\al=4$.
However, since this paper focuses on the exponents $\al$ and $\be$ for which the Cauchy problem is well-posed or ill-posed even in Sobolev spaces of high regularity,
we do not discuss in detail the regularity conditions under which well-posedness holds.
}
In particular, in this case there is no distinction between the problems posed on $\R$ and $\T$.
However, for $\be>0$, there is in general a difference in well-posedness in Sobolev spaces, even at high regularity, between $\R$ and $\T$.

\subsection{On the real-line}
\label{subsec:R}

First, we consider the Cauchy problem for the higher-order semi-linear Schr\"odinger equation:
\begin{equation}
\left\{
\begin{aligned}
&\dt u + i (-\dx^2)^j u = \Nl (u, \dots, \dx^{2j-1} u,
\cj u, \dots, \cj{\dx^{2j-1} u}),
\\
&u|_{t=0} = \phi,
\end{aligned}
\right.
\label{hNLS}
\end{equation}
where $j \in \N$ and $\Nl$ is a polynomial.
Kenig, Ponce, and Vega \cite{KPV93} proved that \eqref{hNLS} with $j=1$ is well-posed in $H^s(\R)$ for $s \ge \frac 72$,
provided that the degree of $\Nl$ is at least three;
see also \cite{Por18,HIT21}.
By combining the local smoothing estimate (Theorem 4.1 in \cite{KPV91}; see also Theorem \ref{thm:locs}) with the maximal function estimate (Corollary 2.9 in \cite{KPV91b}; see also Theorem \ref{thm:KVP2}),
their argument extends to all $j \ge 2$.
In particular, for any $j \in \N$, the equation \eqref{hNLS} is well-posed in $H^s(\R)$ for some (large) $s \in \R$,
provided that $\deg \Nl \ge 3$.
In this regime, apart from the requirement $\deg \Nl \ge 3$,
the well-posedness in $H^s(\R)$ for large $s$ is independent of the specific form of the nonlinearity $\Nl$.
See also \cite{KPV04,LP02,HGM22} for quasi-linear cases.

On the other hand,
if $\Nl$ contains a quadratic derivative term, the situation changes.
In fact, Christ \cite{Chr03} showed that the Cauchy problem
\begin{equation}
\left\{
\begin{aligned}
&\dt u - i \dx^2 u = u \dx u,\\
&u|_{t=0} = \phi
\end{aligned}
\right.
\label{dNLSC}
\end{equation}
is ill-posed in $H^s(\R)$ for any $s \in \R$.
Moreover,
Gr\"unrock \cite{Gr00} showed that the Cauchy problem
\[
\left\{
\begin{aligned}
&\dt u - i \dx^2 u = \cj u \dx \cj u,
\\
&u|_{t=0} = \phi
\end{aligned}
\right.
\]
is well-posed in $H^s(\R)$ for $s \ge 0$.

However, by working in a suitable function space, one can obtain well-posedness for \eqref{hNLS} even in the presence of quadratic derivative nonlinearities.
In fact, Kenig, Ponce, and Vega \cite{KPV93} proved that \eqref{hNLS} with $j=1$ is well-posed in an appropriate weighted Sobolev space when $\deg \Nl \ge 2$.
Moreover, Ozawa \cite{Oz98} showed that \eqref{dNLSC} is well-posed in the space
\[
\Big\{ f \in H^s(\R) \mid \sup_{x \in \R} \Big| \int_0^x f(y) dy \Big| < \infty \Big\}
\]
for $s \ge 1$.
Note that this additional condition is reminiscent of the Mizohata condition appearing in \cite{Mizo85}.
See also \cite{Mizu06,Tar11} for a fourth-order analogue of the Mizohata condition.

Recently, Nakanishi and Wang \cite{NaWa25} proved that the Cauchy problem for very general systems of evolution equations is globally well-posed in the space of distributions supported on the Fourier half-space, denoted by $\Ft^{-1} \X$.
Here, $\X$ is defined as follows.
 
\begin{definition}
\label{def:X}
Let $\D(\R)$ be the space of smooth functions compactly supported on $\R$.
We define $\X$ to be the space of $F \in \D'(\R)$ for which there exist $a \in (0,\infty)$ and a $\C$-valued Borel measure $\fm$ on $\R$ such that
\[
F = \fm \text{ on } (-\infty,a)
\quad
\text{and}
\quad
|\fm|((-\infty,0])=0,
\]
where $|\fm|$ stands for the total variation of $\fm$.
\end{definition}

A more precise definition of $\Ft^{-1} \X$, together with its properties, will be given in Section \ref{sec:NaWa}.
Note that their result \cite{NaWa25} applies in particular to \eqref{dNLS}.

In this paper, we determine the sharp exponents of the fractional derivatives for which the Cauchy problem \eqref{dNLS} is well-posed in Sobolev spaces.
The following theorem is the main result in this paper.

\begin{theorem}
\label{thm:ill}
When $\Ml = \R$,
let $\al, \be \in \R$ satisfy
$\al>0$ and
\begin{equation}
\be
>
\max \Big( \frac{\al-1}2, 0 \Big).
\label{cond:be}
\end{equation}
Let $s, \s \in \R$.
Then, for any $0< \eps \ll 1$, there exist
an initial datum $\phi \in H^s(\R)$ with $ \| \phi \|_{H^s} < \eps$ and $\supp \ft \phi \subset [0, \infty)$
and a time  $T \in (0,\eps)$
such that
the corresponding solution $u \in C([0,T]; \Ft^{-1} \X)$ to \eqref{dNLS} satisfies
\[ 
\| u(T) \|_{H^\s} > \eps^{-1}.
\]
\end{theorem} 

Theorem \ref{thm:ill} establishes norm inflation with infinite loss of regularity, which implies that \eqref{dNLS} is ill-posed in Sobolev spaces whenever \eqref{cond:be} holds.
This generalizes the result of Christ \cite{Chr03}.
Moreover, the condition \eqref{cond:be} is optimal.
Indeed, if $\al$ and $\be$ satisfy
\[
\al>0, \quad
0 \le \be \le
\max \Big( \frac{\al-1}2, 0 \Big).
\]
then \eqref{dNLS} is well-posed in Sobolev spaces.
In other words, on the real line, the local smoothing effect allows one to recover up to $\frac{\alpha - 1}{2}$ derivatives in the quadratic nonlinearity.
The case $\beta = 0$  was mentioned in \eqref{qNLS} before Subsection \ref{subsec:R}.
Although the well-posedness for $\beta > 0$ can be derived from the linear estimates in \cite{KPV91,KPV91b},
we include a proof in Appendix \ref{app:WP} for completeness.

For the proof of Theorem \ref{thm:ill},
we use the global well-posedness in $\Ft^{-1} \X$.
The result in \cite{NaWa25} (see Theorem \ref{thm:wel} below) yields that
the solution $u$ to \eqref{dNLS} can be expanded the sum of iterated terms;
see \eqref{iter0}.
By choosing appropriate initial data $\phi$,
we can prove that the second iterate  $u^{(2)}$ is the leading term.
This yields norm inflation, that is, Theorem \ref{thm:ill} with $\s=s$.
To obtain norm inflation with infinite loss of regularity,
however,
an estimate for $u^{(2)}$ alone is not sufficient.
For a given $\sigma \in \R$,
we instead take a (large) $k \in \N$ and need to prove that the $k$-th iterate $u^{(k)}$ is the leading term in the expansion \eqref{iter0}.
In this setting, it is necessary to extract the main contributing part from each iterated term;
see Proposition \ref{prop:Ftuk}.

The most problematic part of the nonlinear term in \eqref{dNLS} is the interaction in which a low frequency and a high frequency generate an even higher frequency.
On $\T$, one may take the zero mode as the low frequency, and the linear term can then be regarded as the dominant contribution;
we discuss this in more detail in the next subsection.
On the other hand, on $\R$, one must treat a low-frequency component of nonzero width,
and it becomes necessary to estimate each term in the iterated expansion arising from this component.
See \eqref{inter} and \eqref{ukHs2}.

\begin{remark}
\label{rem:dx1}
\rm
The same conclusion as in Theorem \ref{thm:ill} remains valid if, in \eqref{dNLS}, we replace
$iD^\al$ and $D^\be$ by $D^{\al-1} \dx$ and $i D^{\be-1} \dx$, respectively.
In particular,
norm inflation with infinite loss of regularity occurs
for the Cauchy problem for the following equations:
\begin{enumerate}
\item
the $\C$-valued fractional KdV-type equation
\[
\dt u + D^{\al-1} \dx u = u D^\be u
\]
for \eqref{cond:be}.

\item
the quadratic higher-order nonlinear Schr\"odinger equation
\[
\dt u + i (-\dx^2)^j u = u \dx^k u
\]
for $j,k \in \N$ with $j \le k$.
\end{enumerate}
\end{remark}

\subsection{On the torus}

On the torus, even when the dispersive effect is strong and the nonlinearity has high degree,
an additional structural condition is required to obtain well-posedness in Sobolev spaces.
In fact, the authors in \cite{KoOk25} determined necessary and sufficient conditions for \eqref{hNLS} with $j=1$ to be well-posed in $H^s(\T)$ for $s>\frac 52$.
Their ill-posedness result is derived from the non-existence of solutions for certain initial data;
see also \cite{KKO25} for fNLS.

The existence of solutions depends sensitively on the choice of function space.
For instance, the result in \cite{KoOk25} implies that \eqref{dNLSC} is ill-posed in $H^s(\T)$ for $s > \frac 52$ in the sense that solutions fail to exist.
Note that Chihara \cite{Chi02} and Christ \cite{Chr03} proved that the solution map for \eqref{dNLSC} is discontinuous in Sobolev spaces.
In contrast, one can construct solutions to \eqref{dNLSC} in the space of analytic functions.
Furthermore, Chung, Guo, Kwon, and Oh \cite{CGKO17} established global well-posedness of \eqref{dNLSC} in $L^2(\T)$ under mean-zero and smallness assumptions on the initial data.

Moreover, the result of \cite{NaWa25} also applies to \eqref{dNLS} on $\T$,
since 
$\Ft^{-1} \X$ contains periodic distributions as a special case,
where $\X$ is the space defined in Definition \ref{def:X}.
Building on their result, we establish an ill-posedness result for \eqref{dNLS}.
Our main theorem here generalizes our previous work \cite{KoOk25a}, which treated the case $\be=1$;
see also Remark \ref{rem:dx1T2}.

We denote the set of nonnegative integers by $\N_0$.
On $\T$, we obtain the following ill-posedness result.

\begin{theorem} 
\label{thm:illT}
When $\Ml = \T$,
let $\al, \be > 0$ and $s, \s \in \R$.
Then, for any $0<\eps \ll 1$,
there exist an initial datum $\phi \in C^\infty(\T)$ with
$\|\phi\|_{H^s} < \eps$ and $\supp \ft \phi \subset \N_0$,
a time $T \in (0, \eps)$,
and
a function $u$
such that the following conditions hold:
\begin{enumerate}
\item
$u - \ft \phi(0) \in C([0,T]; \Ft^{-1} \X)$.

\item
$u$ is a solution to \eqref{dNLS} in the sense that
\[
\left\{
\begin{aligned}
&\dt \ft u(t,n) + i |n|^\al \ft u(t,n) = \sum_{\substack{n_1,n_2 \in \N_0 \\ n_1+n_2=n}} \ft u(t,n_1) |n_2|^\be \ft u(t,n_2),
\\
& \ft u(0,n) = \ft \phi (n)
\end{aligned}
\right.
\]
for $t \in [0,T]$ and $n \in \N_0$.

\item
$\|u(T)\|_{H^\s} > \eps^{-1}$.
\end{enumerate}
\end{theorem}

Since we require $\ft \phi (0) \neq 0$ in the proof of Theorem \ref{thm:illT},
we can not directly apply the results of \cite{NaWa25};
see Remark \ref{rem:Xn1}.
To avoid any ambiguity regarding the notion of a solution, we state Theorem \ref{thm:illT} (ii).

In contrast to Theorem \ref{thm:ill} on $\R$,
Theorem \ref{thm:illT} establishes norm inflation with infinite loss of regularity for every $\al, \be > 0$. 
In particular, \eqref{dNLS} is ill-posed in Sobolev spaces for all $\al, \be > 0$.
This indicates that, in the periodic setting, there is no local smoothing effect and no recovery of derivative losses in the nonlinearity of \eqref{dNLS}.

The proof of Theorem \ref{thm:illT} is simpler than that of Theorem \ref{thm:ill}.
Let $\phi \in H^s(\T)$ with
$\supp \ft \phi \subset \N_0$.
Suppose that $u$ is a solution to \eqref{dNLS}. 
Set
\begin{align*}
\wt \phi &:= \phi - \ft \phi(0), \\
v &:= u - \ft \phi(0).
\end{align*}
Then, $v$ satisfies the following Cauchy problem:
\begin{equation}
\left\{
\begin{aligned}
&\dt v + i D^\al v = \ft \phi(0) D^\be v + v D^\be v,
\\
&v|_{t=0} = \wt \phi.
\end{aligned}
\right.
\label{dNLS+}
\end{equation}
Note that $\wt \phi \in H^s(\T)$ and $\supp \Ft[ \wt \phi] \subset \N$.

By the result of \cite{NaWa25} (see Theorem \ref{thm:wel-T} below),
the Cauchy problem \eqref{dNLS+} is globally well-posed in $\F^{-1} \X$.
Moreover, if there exists a solution $v$ to \eqref{dNLS+} in $\Ft^{-1} \X$, then by setting $u=v+ \ft \phi (0)$,
we obtain a solution to \eqref{dNLS}.
In the proof of Theorem \ref{thm:illT}, we exploit the fact that $v$ admits an expansion into iterated terms.
Assuming $\supp \Ft[\wt \phi] \neq \emptyset$,
we define $N := \inf \supp \Ft[\wt \phi]$.
Then,
we have
\begin{equation*}
|\ft v (t,N) | 
= |e^{-it N^\al + t \ft \phi (0) N^\be} \Ft[ \wt \phi](N)|
= e^{t (\Re \ft \phi(0)) N^\be} |\ft \phi(N)|.
\end{equation*}
By choosing $\Re \ft \phi (0)$ and $\ft \phi(N)$ appropriately,
we can make the right-hand side grow exponentially in $N$.
This yields norm inflation with infinite loss of regularity, and thus establishes Theorem \ref{thm:illT}.

\begin{remark}
\rm
\label{rem:dx1T2}

\begin{enumerate}
\item
The same comment in Remark \ref{rem:dx1} also applies to Theorem \ref{thm:illT}.

\item
The above argument extends to higher-order nonlinearities.
In particular, 
the same norm inflation result as in Theorem \ref{thm:illT} holds if we replace $u D^\be u$ by $u^k D^\be u$ for any $k \in \N$.
\end{enumerate}
\end{remark}

\begin{remark}
\rm

In the periodic case, when $0 < \be < \al-1$,
one can prove non-existence of solutions to \eqref{dNLS} in $H^s(\T)$ by applying the methods of \cite{KiTs18, KKO25}.
However, it may still be possible that a solution to \eqref{dNLS} exists in a space endowed with a weaker topology for the solution than that of the initial data.
Even in this case, Theorem \ref{thm:illT} shows that the solution map $\phi \mapsto u$ from $H^s(\T)$ to $C([0,T]; H^\s(\T))$ is discontinuous for any $s,\s \in \R$.
\end{remark}

\subsection{Notations}

We denote the set of positive integers by $\N$.
For any set $B$, $\ind_B$ denotes its characteristic function.

For any $x \in \R$,
we define
\begin{align*}
\floor{x} &:= \max \{ n \in \Z \mid n \le x \},
\\
\ceil{x} &:= \min \{ n \in \Z \mid n \ge x \}.
\end{align*}
We write $A \les B$ to mean that $A \le CB$ for some $C > 0$. 
We also write $A \sim B$ when both $ A \les B $ and $ A \ges B $.

Given an integrable function $f$, define
\[
\begin{aligned}
\Ft[f](\xi) &= \ft f(\xi)
:=
\frac 1{\sqrt{2\pi}}
\int_\Ml f(x) e^{-ix \xi} dx.
\end{aligned}
\]
We also denote $\Ft$ and $\ft \cdot$ the Fourier transform of tempered distributions.

We write $f \ast g$ to denote the convolution of $f$ and $g$. 
Set
\[
f^{\ast(0)} = \dl,
\]
where $\dl$ is the Dirac measure concentrated to $0 \in \R$.
We also define $f^{\ast (n)}$ as
\[
f^{\ast(n)} := f^{\ast (n-1)} \ast f
\]
for $n \in \N$.


We fix a $C^\infty$ function $\chi: \R \to [0,1]$ satisfying
$\chi (\xi)=1$ for $\xi \le 1$ and $\chi (\xi)=0$ for $\xi \ge 2$.
Define the smooth cut-off operators as
\begin{equation}
(X^l f)(\xi) := \chi \Big(\frac \xi l \Big) f(\xi), 
\quad (X_l f)(\xi) := \Big( 1 - \chi \Big( \frac \xi l \Big) \Big) f(\xi)
\label{X^_l}
\end{equation}
for $l>0$.

\subsection{Organization of the paper}

The rest of the paper is organized as follows.
In Section \ref{sec:NaWa}, we recall the definition of the space of distributions supported on the Fourier half-space $\Ft^{-1} \X$.
Moreover, we prove that \eqref{dNLS} is well-posed in a subspace of $\Ft^{-1} \X$ by using a contraction mapping argument.
In Section \ref{sec:illR}, after estimating each iterated term, we establish norm inflation with infinite loss of regularity for the real-line case, which yields Theorem \ref{thm:ill}.
In Section \ref{sec:illT}, we show that \eqref{dNLS+} is well-posed in a subspace of $\Ft^{-1} \X$,
and then we prove Theorem \ref{thm:illT}.
In Appendix \ref{app:WP}, we prove that \eqref{dNLS} is well-posed in $H^s(\R)$ when $\al$ and $\be$ do not satisfy condition \eqref{cond:be}, using linear estimates established in \cite{KPV91, KPV91b}.

\section{Global well-posedness in distributions
on the Fourier half-space}
\label{sec:NaWa}

In this section, we recall the notation, the Fourier half-space, and the global well-posedness results due to Nakanishi and Wang \cite{NaWa25}.

\subsection{The space of distributions supported on the Fourier half-space}

Set
\begin{align*}
C^\up &:= \{f:[0,\infty) \to [0,\infty) \mid \text{\rm continuous and non-decreasing} \},
\\
C^\down &:= \{f:[0,\infty) \to [0,\infty) \mid \text{\rm continuous and non-increasing} \}.
\end{align*}
For and any interval $I \subset \R$,
let
\begin{align*}
C^\up_I &:= \{ f \in C^\up \mid f([0,\infty)) \subset I \}, \\
C^\down_I &:= \{ f \in C^\down \mid f([0,\infty)) \subset I \}.
\end{align*}
Moreover,
we define
\[
C^\up_\dl := \big \{f \in C^\up \mid f(0) = 0, \; f([0,2\dl]) \subset \big[0, \frac 12 \big] \big\}
\]
for any $\dl > 0$
and
\[
C^\up_0 := \bigcup_{\dl \in (0,1]} C_\dl^\up.
\]
We denote
\[
\D_{< a} := \{ \psi \in \D (\R) \mid \supp \psi \subset (-\infty, a)\}
\]
for $a \in \R$.

\begin{definition}
For any $m \in [0,\infty)$, we define
\[
\|\psi\|_m^+
:= \max_{\substack{n \in \N_0 \\ n \le m}} \sup_{\xi \in [0, \infty)} |\pd^n \psi(\xi)|.
\]
Moreover, for any $l, m \in [0,\infty)$, and $F \in \D'(\R)$, we define
\begin{equation}
\rho_l^m(F) := \sup \{ \Re \dual{F}{\psi} \mid \psi \in \D_{< l}, \
\| \psi \|_m^+ \le 1 \}.
\label{rholm1}
\end{equation}
\end{definition}

We construct Fourier half space $\X$, using $\rho_l^m$ defined above.
\begin{definition}
Let $\mu, \nu \in C^\up$.
Define
\[
\rho_l^{\mu, \nu} (F) 
:=
\inf \{ B \ge 0 \mid 0 < t \le l \implies \rho_t^{\mu(t)} (F) \le B \nu(t) \}
\]
for $l \in (0, \infty)$ and $F \in \D'(\R)$,
and
\[
\rho^{\mu, \nu} (F) 
:=
\sup_{ l > 0} \rho_l^{\mu, \nu} (F).
\]
Moreover,
we set
\[
\X^{\mu, \nu}
:= \{ F \in \D'(\R) \mid \rho^{\mu, \nu} (F) < \infty \}.
\]
Finally, we define
\begin{equation*}
\X := \bigcup_{\mu,\nu \in C_0^\up} \X^{\mu, \nu}.
\end{equation*} 
\end{definition}

For any $F \in \X^{\mu,\nu}$, $l > 0$, and $\psi \in \D_{<l}$,
we have
\begin{equation}
|\dual{F}{\psi}| \le \nu(l) \rho_l^{\mu, \nu}(F) \|\psi\|_{\mu(l)}^+,
\label{norm}
\end{equation}
where $\rho_l^{\mu, \nu}(F)$ is the optimal non-negative number satisfying this inequality.

\begin{remark}
\rm
\begin{enumerate}
\item
For any $\mu, \nu \in C^\up$,
the space $\X^{\mu,\nu}$ is the Banach space.

\item
The definition of $\X$ above is equivalent to that in Definition \ref{def:X}.
See Lemma 2.4 in \cite{NaWa25}.

\item
The space $\X$ is endowed with the inductive limit topology of locally convex spaces.
See Definition 2.2 in \cite{NaWa25}.
\end{enumerate}
\end{remark}

The Fourier image
\[
\F \D(\R)
:=
\{ \ft \psi \mid \psi \in \D(\R) \}
\]
has the natural topology induced by $\Ft$ from $\D(\R)$.
Then,
$\Ft^{-1} \D'(\R)$ defined as the dual space:
\[
\Ft^{-1} \D'(\R)
:= \{ F : \Ft \D(\R) \to \C \mid \text{linear and continuous} \}.
\]
The Fourier transform $\Ft : \Ft^{-1} \D'(\R) \to \D'(\R)$ is defined by
\[
\dual{\Ft F}{\psi} :=
\dual{F}{\ft \psi}
\]
for $F \in \Ft^{-1} \D'(\R)$ and $\psi \in \D(\R)$.
Moreover,
for any topological linear space $X \hookrightarrow \D'(\R)$,
its Fourier inverse image is denoted by
\[
\Ft^{-1} X
:= \{ F \in \Ft^{-1} \D'(\R) \mid \ft F \in X \},
\]
with the natural topology induced by $\Ft^{-1}$.

We define a solution to \eqref{dNLS}.

\begin{definition}
Let $T > 0$. We say that $u \in C([0,T]; \Ft^{-1} \X)$ is a solution to \eqref{dNLS} when $u$ satisfies
\[
u(t)= e^{-itD^\al} \phi + \int_0^t e^{-i(t-t')D^\al}(u D^\be u(t')) dt'
\]
in $\Ft^{-1} \X$ for all $t \in [0,T]$.
\end{definition}

The result in \cite{NaWa25}
yields the following global well-posedness of \eqref{dNLS} in $\Ft^{-1} \X$.

\begin{theorem}[Corollary 4.2 in \cite{NaWa25}]
\label{thm:NaWa4.2}
The Cauchy problem \eqref{dNLS} is globally well-posed in $\Ft^{-1} \X$.
More precisely, the followings hold:
\begin{enumerate}
\item
For any $\phi \in \Ft^{-1} \X$, there is a unique solution $u \in C([0, \infty); \Ft^{-1} \X)$.
\item 
For any sequence $\{\phi_n\}_{n \in \N} \subset \Ft^{-1} \X$ convergent to some $\phi_\infty \in \Ft^{-1} \X$,
let $u_j \in C([0, \infty); \Ft^{-1} \X)$ be the corresponding solutions for $j \in \N \cup \{\infty\}$. 
Then, for any $T > 0$,  $u_n(t) \to u_\infty (t) \in \Ft^{-1} \X$ uniformly for $0 \le t \le T$ as 
$n \to \infty$.
\end{enumerate}
\end{theorem}

In our setting, we consider \eqref{dNLS} in $H^s(\R)$. The following lemma yields that
we can take $\mu \equiv 0$.

\begin{lemma}
\label{lem:supp}
Let $s \in \R$ and define
\begin{equation}
\nu_0 (t) := \Big( \int_0^t \jb{\xi}^{-2s} d\xi \Big)^{\frac 12}
\label{nu}
\end{equation}
for $t \ge 0$.
Suppose that $F \in H^s(\R)$ satisfies $\supp \ft F \subset [0, \infty)$. 
Then, we have $\ft F \in \X^{0,\nu_0}$.
In particular, $F \in \Ft^{-1} \X$. 
\end{lemma}

\begin{proof}
Note that $\nu_0$ defined in \eqref{nu} belongs to $C_0^\up$.
Let $l > 0$ and $\psi \in \D_{< l}(\R)$.
From $F \in H^s(\R)$, we have $\ft F \in L^2_{loc}(\R)$.
H\"older's inequality, $\supp \ft F \subset [0,\infty)$, and \eqref{nu} imply that
\begin{align*}
|\dual {\ft F}\psi |
&\le
\int_\R |\ft F(\xi)| \ind_{[0,l)} |\psi(\xi)| d\xi
\\
&\le
\Big( \int_0^l \jb{\xi}^{-s} \jb{\xi}^s |\ft F(\xi)| d\xi \Big) \|\psi\|_0^+ 
\\
&\le
\nu_0 (l)
\|F\|_{H^s} \|\psi\|_0^+. 
\end{align*}
With \eqref{norm}, we obtain that
\[
\rho_l^{0, \nu_0}(\ft F) 
\le \|F\|_{H^s}
\]
for $l >0$.
Hence, we have
\[
\rho^{0, \nu_0}(\ft F)
= 
\sup_{l > 0} \rho_l^{0, \nu_0}(\ft F) 
\le 
\|F \|_{H^s} 
< \infty.
\]
Then, $\ft F \in \X^{0, \nu_0}$.
\end{proof}

\subsection{Global well-posedness}

In the proof of Theorem \ref{thm:ill},
we need a result that follows from the proof of Theorem \ref{thm:NaWa4.2}.
In particular, we use the fact that a solution can be obtained by a contraction mapping on a suitable subspace of $\X$.
Although this is essentially a repetition of the argument of \cite{NaWa25},
it is required in our proof, so we summarize their argument in our setting.

\begin{theorem}
\label{thm:wel}
Let $\al,\be >0$.
For any $\nu_0 \in C_0^\up$ and $b_0>0$,
there exists
$\nu \in C_0^\up$
such that the followings hold:

\begin{enumerate}
\item
$\nu(l) \ge \nu_0(l)$ for $l \ge 0$.

\item
For any $\phi \in \Ft^{-1} \X^{0, \nu_0}$
with $\rho^{0, \nu_0}(\ft \phi) \le b_0$,
there exists a unique solution $u \in C([0,1]; \Ft^{-1} \X^{0,\nu})$ to \eqref{dNLS}
satisfying
\[
\sup_{0 \le t \le 1}
\rho^{0, \nu}( \ft u(t) ) \le 2 b_0.
\]
\end{enumerate}
\end{theorem}

The idea of the proof of Theorem \ref{thm:wel} is to solve
the corresponding integral equation to \eqref{dNLS}
as the fixed point of the operator
\begin{equation}
\G(\ft \phi, \ft u) (t,\xi)
:= e^{-it |\xi|^\al} \ft \phi (\xi) + \int_0^t e^{-i(t-t') |\xi|^\al} \Ft[u D^\be u](t', \xi) dt'
\label{ftint}
\end{equation}
on
\[
\Big\{
u \in C([0, 1]; \X^{0, \nu}) \mid
\sup_{0 \le t \le 1} \rho^{0,\nu}(u(t)) \le 2 b_0 \Big\}
\]
for some $\nu \in C^\up_0$.
Note that $\nu$ in Theorem \ref{thm:wel} depends on $\nu_0$ and $b_0$.

The following lemma plays an important role 
to prove Theorem \ref{thm:wel}.

\begin{lemma}
\label{lem:convo}
There exists a mapping
\[
\CC : C_0^\up \times C_{(0,1]}^\down \to C_0^\up
\]
satisfying the following:
For any $\nu_0 \in C_0^\up$, $\kappa \in C_{(0,1]}^\down$, set
\begin{equation*}
\nu_1 := \CC (\nu_0, \kappa).
\end{equation*}
Then, we have the following three conditions:
\begin{enumerate}
\item
$\kappa \nu_1 \in C_0^\up$
and $\nu_0 \le \kappa \nu_1$.

\item
For any $F, G \in \X^{0, \nu_1}$ and $l > 0$, we have
\begin{equation}
\rho^{0, \kappa \nu_1}_l (F \ast G) \le \rho^{0, \nu_1}_l (F) \rho^{0, \nu_1}_l (G).
\label{convo0}
\end{equation}

\item
Let $b \in C^\up_{[1,\infty)}$ satisfy $b(l) = b(0)$ as long as $\nu_0(l) \le \kappa(l)$.
Then,
$\CC ( b \nu_0, \kappa) \ge b \CC (\nu_0,\kappa)$.
\end{enumerate}
\end{lemma}

\begin{proof}
See Lemma 3.4 in \cite{NaWa25} with $\mu_0 \equiv 0$.
Note that $\mu_1 \equiv 0$ by (3.20) in \cite{NaWa25}.
\end{proof}

The key of this lemma for solving \eqref{dNLS} is that the left-hand side of \eqref{convo0}
include the small factor $\kappa$.  

The following lemma states that the boundedness of the linear evolution operator in $\Ft^{-1} \X^{0,\nu}$ for $\nu \in C_0^\up$.

\begin{lemma}
\label{lem:bd1}
Let $\nu \in C_0^\up$, $F \in \X^{0,\nu}$, $\al>0$, and $t \in \R$.
Set
\[
M(F):= e^{it |\cdot|^\al} F.
\]
Then, we have $M(F) \in \X^{0,\nu}$ and
\[
\rho_l^{0,\nu}(M(F))
= \rho_l^{0,\nu}(F)
\]
for any $l>0$.
\end{lemma}

\begin{proof}
For $l>0$ and $\psi \in \D_{<l}$,
we have
\[
\dual{M(F)}{\psi}
=\lim_{\dl \to +0}
\dual{F}{X_\dl M(\psi)},
\]
where $X_\dl$ is defined in \eqref{X^_l}.
It follows from \eqref{norm} that
\[
|\dual{F}{X_\dl M(\psi)}|
\le
\nu (l) \rho_l^{0,\nu}(F) \| X_\dl M(\psi) \|_0^+
\le
\nu (l) \rho_l^{0,\nu}(F) \| \psi \|_0^+
\]
for any $l>0$ and $\dl>0$.
Hence, we obtain that
\[
\rho_l^{0,\nu}(M(F))
\le \rho_l^{0,\nu}(F)
\]
for $l>0$.
Moreover,
we set $\wt M(F) := e^{-it |\cdot|^\al} F$.
Since $F = \wt M (M(F))$,
we have
\[
\rho_l^{0,\nu}(F)
=
\rho_l^{0,\nu} ( \wt M (M(F)))
\le
\rho_l^{0,\nu} (M(F)).
\]
Therefore, we obtain that
\[
\rho_l^{0,\nu} (M(F))
=
\rho_l^{0,\nu}(F),
\]
which concludes the proof.
\end{proof}

\begin{proof}[Proof of Theorem \ref{thm:wel}]
We fix $\eps > 0$ small enough such that
\begin{equation}
\eps < \frac 1{8 b_0}.
\label{eps1}
\end{equation}
Define $\nu_\ast, \nu_1, \nu \in C_0^\up$ and $\kappa \in C^\down_{(0,1]}$ as 
\begin{equation}
\begin{aligned}
&\nu_\ast (l) := \eps^{-1} \nu_0(l), &&
\kappa (l) := \frac 1{1 + l^\be},
\\
&\nu_1(l) := \CC (\nu_\ast, \kappa) (l), &&
\nu (l) := \eps \kappa (l) \nu_1(l)
\end{aligned}
\label{nusR}
\end{equation}
for $l \ge 0$,
where $\CC$ is defined in Lemma \ref{lem:convo}. 

Set
\begin{equation}
Z := \Big\{\ft u \in C([0, 1]; \X^{0,\nu}) \mid \sup_{0 \le t \le 1} \rho^{0, \nu} (\ft u(t)) \le 2 b_0 \Big\}.
\label{spite}
\end{equation}
Let $\phi \in \Ft^{-1} \X^{0, \nu_0}$
satisfy $\rho^{0, \nu_0}(\ft \phi) \le b_0$.
We will show that
$\G(\ft \phi, \cdot)$ defined in \eqref{ftint} is a contraction mapping on $Z$.

For any $\ft u \in C([0,1], \X^{0,\nu})$ and $\psi \in \D_{<l}$,
it follows from \eqref{norm} and Lemma \ref{lem:bd1} that
\begin{equation*}
\begin{aligned}
\Re \dual{\G(\ft \phi, \ft u) (t)}{\psi} 
&\le
| \dual{e^{-it|\cdot|^\al} \ft \phi}{\psi} | + 
\Big|\Dual {\int_0^t e^{-i(t-t')|\cdot|^\al} \Ft[u D^\be u](t') dt'}{\psi}\Big| 
\\
&\le
\nu_0 (l) \rho_l^{0,\nu_0} ( \ft \phi) \| \psi \|_0^+
\\
&\quad
+
\int_0^t
\kappa (l) \nu_1 (l) \rho_l^{0,\kappa \nu_1}(\Ft [u D^\be u](t'))
\| \psi \|_0^+
dt'
.
\end{aligned}
\end{equation*}
With \eqref{rholm1} and $\rho^{0,\nu_0}(\ft \phi) \le b_0$, we obtain
\begin{equation}
\begin{aligned}
\rho_l^0 (\G(\ft \phi, \ft u) (t))
\le
\nu_0 (l) b_0 
+
\int_0^t
\kappa(l) \nu_1(l)
\rho_l^{0, \kappa \nu_1}(\Ft [u D^\be u](t')) dt'.
\end{aligned}
\label{estDuh}
\end{equation}

From \eqref{convo0} and \eqref{eps1}, we have
\begin{equation}
\begin{aligned}
\rho_l^{0, \kappa \nu_1}(\Ft [u D^\be u](t))
&=
\rho_l^{0, \kappa \nu_1} ( \ft u(t) \ast (|\cdot|^\be \ft u(t)) )
\\
&\le
\rho_l^{0, \nu_1} (\ft u(t)) \rho_l^{0, \nu_1} (|\cdot|^\be \ft u(t)).
\end{aligned}
\label{convo}
\end{equation}
By \eqref{norm} and \eqref{nusR}, we obtain
\begin{align*}
| \dual {\ft u(t)}{\psi} |
&\le
\nu(l) \rho_l^{0,\nu} (\ft u (t)) \|\psi\|_{0}^+
\\
&=
\eps \kappa (l) \nu_1 (l) \rho_l^{0,\nu} (\ft u (t)) \|\psi\|_0^+
\\
&\le
\eps \nu_1(l) \rho_l^{0,\nu} (\ft u (t)) \|\psi\|_0^+
\end{align*}
for $l>0$ and $\psi \in \D_{<l}$.
Hence,
it holds that
\begin{equation}
\rho_l^{0,\nu_1} (\ft u(t))
\le
\eps \rho_l^{0,\nu} (\ft u (t)).
\label{udual}
\end{equation}
Moreover, \eqref{norm} and \eqref{nusR} yield that
\begin{align*}
| \dual {|\cdot|^\be \ft u(t)}{\psi} |
&\le
\nu(l) \rho_l^{0,\nu} (\ft u(t)) \big\| |\cdot|^\be \psi \big\|_0^+
\\
&\le
l^\be \eps \kappa (l) \nu_1 (l) \rho_l^{0,\nu} (\ft u (t)) \|\psi\|_0^+
\\
&\le
\eps \nu_1(l) \rho_l^{0,\nu} (\ft u (t)) \|\psi\|_0^+
\end{align*}
for $l>0$ and $\psi \in \D_{<l}$.
Hence, it holds that
\begin{equation}
\rho_l^{0,\nu_1}(|\cdot|^\be \ft u(t))
\le
\eps \rho_l^{0,\nu} (\ft u (t)).
\label{xiudua}
\end{equation}
From \eqref{convo}--\eqref{xiudua},
we have
\begin{equation}
\rho_l^{0, \kappa \nu_1}(\Ft [u D^\be u](t))
\le
\eps^2 \big(\rho_l^{0, \nu} (\ft u(t)) \big)^2
\label{estcon}
\end{equation}
for any $l>0$ and $t \in [0,1]$.

It follows from \eqref{estDuh}, \eqref{estcon}, and \eqref{nusR}
that
\begin{equation*}
\begin{aligned}
\rho_l^0 (\G(\ft \phi, \ft u) (t))
&\le
\nu_0 (l) b_0 
+ 
\int_0^t \kappa(l) \nu_1(l)
\eps^2 \big(\rho^{0, \nu} (\ft u(t')) \big)^2 dt'
\\
&=
\nu_0 (l) b_0 
+
\eps \int_0^t \nu(l) \big(\rho^{0, \nu} (\ft u(t')) \big)^2 dt'.
\end{aligned}
\end{equation*}
Lemma \ref{lem:convo} (i) and \eqref{nusR} yield that
\[
\nu
= \eps \kappa \nu_1
\ge \eps \nu_\ast
= \nu_0.
\]
Hence,
we have
\[
\rho^{0, \nu}(\G(\ft \phi, \ft u) (t))
\le
b_0 + 4 \eps b_0^2.
\]
for $\ft u \in Z$ and $t \in [0,1]$,
where $Z$ is defined in \eqref{spite}.
By \eqref{eps1},
we obtain that
$\G(\ft \phi, \cdot)$ is a mapping on $Z$.

Let $\ft u_1, \ft u_2 \in Z$.
The same calculation as in \eqref{estDuh} yields that
\begin{equation}
\begin{aligned}
&\rho_l^0 (\G (\ft \phi, \ft u_1) (t) - \G (\ft \phi, u_2)(t))
\\
&\le
\int_0^t
\kappa(l) \nu_1(l)
\rho^{0, \kappa \nu_1}(\Ft [u_1 D^\be u_1 - u_2 D^\be u_2](t')) dt'.
\end{aligned}
\label{difDuh}
\end{equation}
By similar calculations as in \eqref{convo}--\eqref{estcon}, we have
\begin{equation}
\begin{aligned}
&\rho^{0, \kappa \nu_1}( \Ft [u_1 D^\be u_1 - u_2 D^\be u_2](t))
\\
&=
\rho^{0, \kappa \nu_1} (\Ft [(u_1 - u_2) D^\be u_1](t)) 
+ \rho^{0, \kappa \nu_1} (\Ft [u_2 D^\be (u_1 - u_2)](t)) 
\\
&\le
4 b_0 \eps^2 \rho^{0, \nu} (\ft u_1 (t) - \ft u_2 (t)).
\end{aligned}
\label{difcv}
\end{equation}
It follows \eqref{difDuh} and \eqref{difcv} that
\[
\sup_{0 \le t \le 1} \rho^{0, \nu} (\G (\ft \phi, \ft u_1) (t) - \G (\ft \phi, \ft u_2) (t))
\le
4 b_0 \eps \sup_{0 \le t \le 1} \rho^{0, \nu} (\ft u_1 (t) - \ft u_2 (t))
\]
for any $\ft u_1, \ft u_2 \in Z$.
By \eqref{eps1},
we obtain that
$\G (\ft \phi, \cdot)$ is a contraction mapping on $Z$.
Therefore,
there exists a unique solution $u$ to \eqref{dNLS} in $\Ft^{-1} Z$.
\end{proof}

%
\section{Ill-posedness on the real-line}\label{sec:illR}

In this section, we prove Theorem \ref{thm:ill}.
From \eqref{cond:be},
let $\ta$ satisfy
\begin{equation}
\max \Big( 0 , \al -1, \frac 23 (\be-1) \Big)
<\ta
< 2 \be .
\label{d:ta1}
\end{equation}
For $s \in \R$, $N \in \N$, and $0 < \eps \ll 1$, we set
\begin{equation}
\begin{aligned}
\phi 
&= \eps N^{\frac {\ta}2} \Ft^{-1} \big[ \ind_{[\frac 1{N^{\ta}}, \frac 2{N^{\ta}}]} \big]
+ \eps N^{- s+\frac {\ta}2} \Ft^{-1} \big[ \ind_{[N, N + \frac 1{N^{\ta}}]} \big]
\\
&:=
\phi_1 + \phi_2.
\end{aligned}
\label{defphi}
\end{equation}
A direct calculation yields that
\begin{equation}
\begin{aligned}
\|\phi\|_{H^s}
&\le
\eps \Big(\int_\R \jb{\xi}^{2s} \big(N^{\ta} \ind_{[\frac 1{N^{\ta}}, \frac 2{N^{\ta}}]}(\xi) 
+ N^{-2 s + \ta} \ind_{[N, N+\frac 1{N^{\ta}}]}(\xi) \big) d\xi \Big)^{\frac 12}
\\
&\les
\eps.
\end{aligned}
\label{phiHs}
\end{equation}
Thus, we obtain $\phi \in H^s(\R)$.
It follows from Lemma \ref{lem:supp} that $\phi \in \F^{-1} \X^{0,\nu_0}$,
where $\nu_0$ is defined in \eqref{nu}.

Set 
\begin{equation}
u^{(1)}(t) := e^{-itD^\al} \phi.
\label{defu1}
\end{equation}
For $k \ge 2$, we define
\begin{equation}
u^{(k)}(t):= \sum_{\substack {k_1, k_2 \in \N \\ k_1 + k_2 = k}} 
\int_0^t e^{-i(t-t')D^\al} (u^{(k_1)} D^\be u^{(k_2)}) dt'.
\label{defuk}
\end{equation}
By Theorem \ref{thm:wel},
there exists $\nu \in C_0^\up$ and
a solution
\[
u \in C([0,1]; \Ft^{-1} \X^{0,\nu})
\]
to \eqref{dNLS}.
Moreover,
$u$ is written as
\begin{equation}
u= \sum_{k=1}^\infty u^{(k)}.
\label{iter0}
\end{equation}
in $\Ft^{-1} Z$, where $Z$ defined in \eqref{spite} with $b_0 = \rho^{0,\nu_0}(\ft \phi)$.

For $\supp \Ft[u^{(k)}] (t,\cdot)$, the following lemma holds.

\begin{lemma}
\label{lem:suppuk}
Let $k \in \N$. For any $t \in [0, \infty)$, we have
\begin{equation}
\supp \Ft[u^{(k)}](t,\cdot) \subset 
\bigcup_{\substack{j_1, j_2 \in \N_0 \\ j_1 + j_2 = k}}
\Big[ j_2 N + \frac{j_1}{N^{\ta}} , \; j_2 N + \frac{k+ j_1}{N^{\ta}} \Big]. 
\label{suppuk}
\end{equation}
\end{lemma}

\begin{proof}
We prove \eqref{suppuk} by an induction argument.
When $k=1$, it follows from \eqref{defphi} and \eqref{defu1} that \eqref{suppuk} holds.
Suppose that \eqref{suppuk} holds up to $k-1$ for $k \ge 2$.
Let $k_1, k_2 \in \N$ satisfy $k_1+k_2=k$.
By assumption,
\begin{align*}
\supp \Ft[u^{(k_j)}](t,\cdot)
&\subset
\bigcup_{l_j=0}^{k_j}
\Big[ l_j N + \frac {k_j -l_j}{N^{\ta}}, \; l_j N + \frac{2 k_j - l_j}{N^{\ta}} \Big]
\end{align*}
for $j=1,2$.
It follows that
\begin{equation*}
\begin{aligned}
&\supp
\F \big[ u^{(k_1)} D^\be u^{(k_2)} \big] (t, \cdot)
\\
&\subset
\bigcup_{\substack{0 \le l_1 \le k_1 \\ 0 \le l_2 \le k_2}} 
\Big[ (l_1 + l_2) N + \frac{k - (l_1 + l_2)}{N^{\ta}}, (l_1 + l_2) N + \frac{2k - (l_1 + l_2)}{N^{\ta}} \Big]
\\
&=
\bigcup_{\substack{j_1, j_2 \in \N_0 \\ j_1 + j_2 = k}}
\Big[ j_2 N + \frac{j_1}{N^{\ta}} , \; j_2 N + \frac{k+ j_1}{N^{\ta}} \Big].
\end{aligned}
\end{equation*}
With \eqref{defuk},
we have
\begin{equation*}
\begin{aligned}
\supp \Ft[u^{(k)}](t,\cdot)
&\subset
\bigcup_{\substack {k_1, k_2 \in \N \\ k_1 + k_2 = k}}
\F \big[ u^{(k_1)} D^\be u^{(k_2)} \big] (t, \cdot)
\\
&\subset
\bigcup_{\substack{j_1, j_2 \in \N_0 \\ j_1 + j_2 = k}}
\Big[ j_2 N + \frac{j_1}{N^{\ta}} , \; j_2 N + \frac{k+ j_1}{N^{\ta}} \Big]
\end{aligned}
\end{equation*}
for any $t \in [0,\infty)$.
This yields that \eqref{suppuk} holds.
\end{proof}

We also define
\begin{equation}
u_1^{(1)}(t) := e^{-itD^\al} \phi_1,
\label{defuk11}
\end{equation}
where $\phi_1$ is defined in \eqref{defphi}.
For $k \ge 2$, we define
\begin{equation}
u_1^{(k)}(t):= \sum_{\substack {k_1, k_2 \in \N \\ k_1 + k_2 = k}} 
\int_0^t e^{-i(t-t')D^\al} (u_1^{(k_1)} D^\be u_1^{(k_2)}) dt'.
\label{defuk12}
\end{equation}
Let $k \in \N$ and $\xi \in \big[ N + \frac {k-1}{N^{\ta}}, N + \frac k{N^{\ta}} \big)$.
Lemma \ref{lem:suppuk} with \eqref{iter0} yields that 
\begin{equation}
\begin{aligned}
\ft u(t,\xi)
&= \sum_{l = \floor {\frac k2}}^k 
\Ft [u^{(l)}](t,\xi)
+
\sum_{l = \floor{\frac{k-1+N^{\ta + 1}}2}}^{\ceil{k+N^{\ta + 1}}} \Ft[u_1^{(l)}](t,\xi)
\end{aligned}
\label{inter}
\end{equation}
for $N \gg 1$.

We estimate $|\Ft [u_1^{(k)}] (t,\xi)|$ for any $k \in \N$.

\begin{proposition}
\label{prop:u1k}
For $t > 0$, $\xi \in \R$, $k \in \N$, and $N \in \N$,
we have
\[
\begin{aligned}
| \Ft[u_1^{(k)}](t,\xi)|
&\le
\eps
(\pi^2 2^\be \eps t )^{k-1}
((k-1)!)^{\max (0, \be-1)}
\\
&\quad
\times
N^{\ta \be + \ta ( -\be + \frac 12) k}
\ind_{[\frac 1{N^{\ta}}, \frac {2}{N^{\ta}}]}^{\ast (k)}(\xi).
\end{aligned}
\]
\end{proposition}

For the proof, we use the following lemma.

\begin{lemma}
\label{lem:bdda}
Let $\{ a_k \}_{k \in \N}$ be a sequence of positive real numbers.
Assume that there exists $C_0>0$ such that
\[
a_k \le C_0 \sum_{\substack{k_1,k_2 \in \N \\ k_1+k_2=k}} a_{k_1} a_{k_2}
\]
holds for $k \ge 2$.
Then, we have
\[
a_k \le \Big( \frac 23 \pi^2 C_0 \Big)^{k-1} a_1^k
\]
for any $k \in \N$.
\end{lemma}

See, for example, Lemma 4.2 in \cite{MaOk16}.
See also Lemma 3.5 in \cite{Kis19}.

\begin{proof}[Proof of Proposition \ref{prop:u1k}]
Let $\{ a_k \}_{k \in \N}$ be a sequence of positive real numbers such that
$a_1=1$ and
\begin{equation}
a_k
:=
\sum_{\substack{k_1,k_2 \in \N \\ k_1+k_2=k}}
\frac{(2 k_2)^\be}{k-1} a_{k_1} a_{k_2}
\label{seq:ak1}
\end{equation}
for $k \ge 2$.
We show that
\begin{equation}
\begin{aligned}
| \Ft[u_1^{(k)}](t,\xi)|
&\le
a_k
\eps
(\eps t )^{k-1}
N^{\ta \be + \ta ( -\be + \frac 12) k}
\ind_{[\frac 1{N^{\ta}}, \frac {2}{N^{\ta}}]}^{\ast (k)}(\xi)
\end{aligned}
\label{u1bd2}
\end{equation}
for $t >0$, $\xi \in \R$, and $k \in \N$.

When $k=1$,
\eqref{u1bd2} holds by \eqref{defuk11} and \eqref{defphi}.
Suppose that \eqref{u1bd2} holds up to $k-1$ for $k \ge 2$.
It follows from \eqref{defuk12}, \eqref{seq:ak1}, and the induction hypothesis that
\begin{equation*}
\begin{aligned}
&|\Ft[u_1^{(k)}](t,\xi)|
\\
&\le
\sum_{\substack {k_1, k_2 \in \N \\ k_1 + k_2 = k}}
\int_0^t \int_\R |\Ft[u_1^{(k_1)}](t',\xi - \eta) \cdot \eta^\be \Ft[u_1^{(k_2)}](t', \eta)| d \eta dt'
\\
&\le
\sum_{\substack {k_1, k_2 \in \N \\ k_1 + k_2 = k}}
\int_0^t \int_\R
a_{k_1}
\eps
( \eps t')^{k_1-1}
N^{\ta \be + \ta ( -\be + \frac 12) k_1}
\ind_{[\frac 1{N^{\ta}}, \frac {2}{N^{\ta}}]}^{\ast (k_1)}(\xi-\eta)
\\
&\hspace*{50pt} \times
\eta^\be
a_{k_2}
\eps
( \eps t')^{k_2-1}
N^{\ta \be + \ta ( -\be + \frac 12) k_2}
\ind_{[\frac 1{N^{\ta}}, \frac {2}{N^{\ta}}]}^{\ast (k_2)}(\eta)
d\eta dt'
\\
&\le
a_k \eps (\eps t)^{k-1}
N^{\ta \be + \ta ( -\be + \frac 12) k}
\ind_{[\frac 1{N^{\ta}}, \frac {2}{N^{\ta}}]}^{\ast (k)}(\xi)
.
\end{aligned}
\end{equation*}
Hence, we obtain \eqref{u1bd2} for $k \in \N$.

From \eqref{u1bd2},
it remains to show that
\begin{equation}
a_k
\le
( \pi^2 2^\be )^{k-1}
((k-1)!)^{\max (0,\be-1)}
\label{eq:adbbb0}
\end{equation}
for $k \in \N$.
Set
\[
b_k := \frac{a_k}{((k-1)!)^{\max (0,\be-1)}}.
\]
With \eqref{seq:ak1}, we have
\begin{align*}
b_k
&=
2^\be
\sum_{\substack{k_1,k_2 \in \N \\ k_1+k_2=k}}
\frac{k_2^\be}{k-1} 
\Big( \frac{(k_1-1)! \cdot (k_2-1)!}{(k-1)!} \Big)^{\max (0,\be-1)}
b_{k_1} b_{k_2}
\\
&\le
2^\be
\sum_{\substack{k_1,k_2 \in \N \\ k_1+k_2=k}}
\frac{k_2}{k-1}
\Big( \frac{(k_1-1)! \cdot k_2!}{(k-1)!} \Big)^{\max (0,\be-1)}
b_{k_1} b_{k_2}.
\end{align*}
When $k_1, k_2 \in \N$ with $k_1 + k_2 = k$,
we have
\[
(k_1-1)! \cdot k_2!
\le (k-1)!.
\]
Hence, it holds that
\begin{align*}
b_k
&\le
2^\be
\sum_{\substack{k_1,k_2 \in \N \\ k_1+k_2=k}}
b_{k_1} b_{k_2}.
\end{align*}
Lemma \ref{lem:bdda} with $C_0= 2^\be$ yields that
\[
b_k
\le
( \pi^2 2^\be )^{k-1}.
\]
Therefore, we obtain that
\[
a_k
\le
( \pi^2 2^\be )^{k-1}
((k-1)!)^{\max (0,\be-1)},
\]
which shows \eqref{eq:adbbb0}.
This concludes the proof of Proposition \ref{prop:u1k}.
\end{proof}

The following proposition is the key to get the lower bound of the first term in \eqref{inter}. 

\begin{proposition}
\label{prop:Ftuk}
Let $\al>0$ and $k \in \N$.
Then,
we have
\begin{equation}
\begin{aligned}
&\bigg|
\Ft[u^{(k)}](t,\xi) 
\\
&\quad
-
e^{-it\xi^\al} 
\frac {\eps^k t^{k-1}}{(k-1)!} N^{- s+ (\frac \ta2 + \be) k- \be}
\big( \ind_{[\frac 1{N^{\ta}}, \frac 2{N^{\ta}}]}^{\ast(k-1)} 
\ast \ind_{[N, N+\frac 1{N^{\ta}}]} \big) (\xi)
\bigg|
\\
&\les
\max \big( N^{-(\ta + 1)}, N^{-(\ta + 1) \be}, t \big)
\eps^k t^{k-1} N^{-s + (\frac {\ta}2 + \be) k - \be}
\\
&\hspace*{150pt}
\times
\big( \ind_{[\frac 1{N^{\ta}}, \frac 2{N^{\ta}}]}^{\ast(k-1)} 
\ast \ind_{[N, N+\frac 1{N^{\ta}}]}\big) (\xi)
\end{aligned}
\label{orduk}
\end{equation}
for $N \gg 1$, $0< t \ll 1$,
and $\xi \in [ N + \frac{k-1}{N^{\ta}}, N + \frac {2k-1}{N^{\ta}}]$.
Here, the implicit constant in \eqref{orduk} depends on $k$.
\end{proposition}

\begin{proof}
We prove \eqref{orduk} by an induction argument.
From \eqref{defu1} and \eqref{defphi}, we have that \eqref{orduk} with $k=1$ holds.
Suppose that \eqref{orduk} holds up to $k-1$ for $k \ge 2$.

Set
\[
u^{(k)}_2 := u^{(k)} - u^{(k)}_1.
\]
When $N \gg 1$
and $\xi \in [ N + \frac{k-1}{N^{\ta}}, N + \frac {2k-1}{N^{\ta}}]$,
we have
\begin{equation}
\begin{aligned}
\Ft[ u^{(k)}] (t,\xi)
&=\Ft[ u^{(k)}_2] (t,\xi)
\\
&=
\sum_{l=1}^{k-1}
\int_0^t e^{-i(t-t') \xi^\al} \Ft \big[ u_1^{(l)} D^\be u_2^{(k-l)} \big](t', \xi) dt'
\\
&\quad
+
\sum_{l=1}^{k-1}
\int_0^t e^{-i(t-t') \xi^\al} \Ft \big[ u_2^{(l)} D^\be u_1^{(k-l)} \big](t', \xi) dt'
\\
&=
\int_0^t e^{-i(t-t') \xi^\al} \Ft \big[ u_1^{(1)} D^\be u_2^{(k-1)} \big](t', \xi) dt'
\\
&\quad
+
\sum_{l=2}^{k-1}
\int_0^t e^{-i(t-t') \xi^\al} \Ft \big[ u_1^{(l)} D^\be u_2^{(k-l)} \big](t', \xi) dt'
\\
&\quad
+
\sum_{l=1}^{k-1}
\int_0^t e^{-i(t-t') \xi^\al} \Ft \big[ u_2^{(l)} D^\be u_1^{(k-l)} \big](t', \xi) dt'
\\
&=:
\I (t,\xi)
+ \II (t,\xi)
+ \III (t,\xi).
\end{aligned}
\label{Ftukaa}
\end{equation}

First, we prove that
\begin{equation}
\begin{aligned}
&\bigg|
\I(t,\xi)
-
e^{-it\xi^\al}
\frac {\eps^k t^{k-1}}{(k-1)!} N^{- s+ (\frac \ta2 + \be) k- \be}
\big( \ind_{[\frac 1{N^{\ta}}, \frac 2{N^{\ta}}]}^{\ast(k-1)} 
\ast \ind_{[N, N+\frac 1{N^{\ta}}]} \big) (\xi)
\bigg|
\\
&\les
\max \big( N^{-(\ta + 1)}, N^{-(\ta + 1) \be}, t \big)
\eps^{k} t^{k-1}
N^{-s + (\frac {\ta}2 + \be) k - \be}
\\
&\hspace*{150pt}
\times
\big( \ind_{[\frac 1{N^{\ta}}, \frac 2{N^{\ta}}]}^{\ast(k-1)} 
\ast \ind_{[N, N+\frac 1{N^{\ta}}]}\big) (\xi)
\end{aligned}
\label{ordukaa}
\end{equation}
for $0< t \ll 1$ and $\xi \in [ N + \frac{k-1}{N^{\ta}}, N + \frac {2k-1}{N^{\ta}}]$.
Set
\[
\begin{aligned}
\I_1(t,\xi)
&:=
e^{-it\xi^\al}
\int_0^t \int_\R e^{it'(\xi^\al -(\xi-\eta)^\al - \eta^\al)} 
\eps N^{\frac {\ta}2} \ind_{[\frac 1{N^{\ta}}, \frac 2{N^{\ta}}]}(\xi-\eta)
\\
&\hspace*{50pt}
\times
\eta^\be
\frac {\eps^{k-1} {t'}^{k-2}}{(k-2)!} N^{- s + (\frac \ta 2 + \be) (k-1) -\be} 
\\
&\hspace*{50pt}
\times
\big( \ind_{[\frac 1{N^{\ta}}, \frac 2{N^{\ta}}]}^{\ast(k-2)} 
\ast \ind_{[N, N+\frac 1{N^{\ta}}]} \big) (\eta) 
d\eta dt'.
\end{aligned}
\]
It follows from
\eqref{Ftukaa}, \eqref{defuk11}, and the induction hypothesis
that
\begin{equation}
\begin{aligned}
|\I(t,\xi) - \I_1 (t,\xi)|
&\les
\int_0^t \int_\R
\eps N^{\frac {\ta}2} \ind_{[\frac 1{N^{\ta}}, \frac 2{N^{\ta}}]}(\xi-\eta)
\\
&\hspace*{50pt}
\times
\eta^\be
\max \big( N^{-(\ta + 1)}, N^{-(\ta + 1) \be}, t' \big)
\\
&\hspace*{50pt}
\times
\eps^{k-1} t'^{k-2} 
N^{- s + (\frac \ta 2 + \be) (k - 1) -\be}
\\
&\hspace*{50pt}
\times
\big( \ind_{[\frac 1{N^{\ta}}, \frac 2{N^{\ta}}]}^{\ast(k-2)} 
\ast \ind_{[N, N+\frac 1{N^{\ta}}]}\big) (\eta)
d\eta dt'
\\
&\les
\max \big( N^{-(\ta + 1)}, N^{-(\ta + 1) \be}, t \big)
\eps^{k} t^{k-1}
\\
&\quad
\times
N^{-s+(\frac \ta2+\be)k -\be}
\big( \ind_{[\frac 1{N^{\ta}}, \frac 2{N^{\ta}}]}^{\ast(k-1)} 
\ast \ind_{[N, N+\frac 1{N^{\ta}}]} \big) (\xi).
\end{aligned}
\label{Idec1}
\end{equation}
We decompose $\I_1$ into two parts:
\begin{equation}
\begin{aligned}
\I_1 (t,\xi)
&=
e^{-it\xi^\al}
\int_0^t \int_\R
\eps N^{\frac {\ta}2} \ind_{[\frac 1{N^{\ta}}, \frac 2{N^{\ta}}]}(\xi-\eta)
\\
&\hspace*{30pt}
\times
\eta^\be
\frac {\eps^{k-1} {t'}^{k-2}}{(k-2)!} N^{- s + (\frac \ta 2 + \be) (k-1) -\be} 
\\
&\hspace*{30pt}
\times
\big( \ind_{[\frac 1{N^{\ta}}, \frac 2{N^{\ta}}]}^{\ast(k-2)} 
\ast \ind_{[N, N+\frac 1{N^{\ta}}]} \big) (\eta) 
d\eta dt'
\\
&\quad
+
e^{-it\xi^\al}
\int_0^t \int_\R
\big( e^{it'( \xi^\al -(\xi-\eta)^\al - \eta^\al)} -1 \big)
\\
&\hspace*{30pt}
\times
\eps N^{\frac {\ta}2} \ind_{[\frac 1{N^{\ta}}, \frac 2{N^{\ta}}]}(\xi-\eta)
\eta^\be
\frac {\eps^{k-1} {t'}^{k-2}}{(k-2)!} N^{- s + (\frac \ta 2 + \be) (k-1) -\be} 
\\
&\hspace*{30pt}
\times
\big( \ind_{[\frac 1{N^{\ta}}, \frac 2{N^{\ta}}]}^{\ast(k-2)} 
\ast \ind_{[N, N+\frac 1{N^{\ta}}]} \big) (\eta) 
d\eta dt'
\\
&=:
\I_{1,1}(t,\xi)
+
\I_{1,2}(t,\xi).
\end{aligned}
\label{I1dec1}
\end{equation}
It follows from
the fundamental theorem of calculus
that
\[
|\eta^\be - N^{\be}| 
= \be (\eta -N)
\int_0^1 (N + \tau (\eta-N))^{\be-1} d\tau
\les 
N^{- \ta + \be - 1}
\]
for $ \eta \in [N + \frac{k-1}{N^{\ta}}, N + \frac {2(k-1)}{N^{\ta}}]$.
With \eqref{I1dec1}, we have
\begin{equation}
\begin{aligned}
&\bigg|
\I_{1,1}(t,\xi)
-
e^{-it\xi^\al}
\frac {\eps^{k} {t}^{k-1}}{(k-1)!} N^{- s + (\frac \ta 2 + \be) k - \be} 
\big( \ind_{[\frac 1{N^{\ta}}, \frac 2{N^{\ta}}]}^{\ast(k-1)} 
\ast \ind_{[N, N+\frac 1{N^{\ta}}]} \big) (\xi)
\bigg|
\\
&\les
N^{-(\ta + 1)}
\cdot
\eps^k t^{k-1}
N^{- s + (\frac \ta 2 + \be) k - \be} 
\big( \ind_{[\frac 1{N^{\ta}}, \frac 2{N^{\ta}}]}^{\ast(k-1)} 
\ast \ind_{[N, N+\frac 1{N^{\ta}}]} \big) (\xi)
.
\end{aligned}
\label{I11aa}
\end{equation}
Moreover,
the fundamental theorem of calculus with \eqref{d:ta1} and $\al>0$ yields that
\begin{equation*}
\begin{aligned}
|\xi^\al -(\xi-\eta)^\al - \eta^\al |
&\le
(\xi-\eta)^\al + (\xi-\eta) \al \int_0^1 (\eta + \tau (\xi -\eta))^{\al-1} d\tau
\\
&\les
N^{-\al \ta}
+ \frac 1{N^{\ta}} N^{\al-1}
\les 1
\end{aligned}
\end{equation*}
for $\eta \in [N + \frac {k-1}{N^{\ta}}, N + \frac {2(k-1)}{N^{\ta}}]$
and $\xi -\eta \in [ \frac 1{N^\ta}, \frac 2{N^\ta} ]$
.
From \eqref{I1dec1}, we have
\begin{equation}
\begin{aligned}
|\I_{1,2}(t,\xi)|
&\les
\int_0^t
t'
\eps N^{\frac \ta2}
N^\be \eps^{k-1} t'^{k-2}
N^{- s + (\frac \ta 2 + \be) (k-1) - \be}
dt'
\\
&\hspace*{30pt}
\times
\big( \ind_{[\frac 1{N^{\ta}}, \frac 2{N^{\ta}}]}^{\ast(k-1)} 
\ast \ind_{[N, N+\frac 1{N^{\ta}}]} \big) (\xi)
\\
&\les
(\eps t)^{k}
N^{- s + (\frac \ta 2 + \be) k - \be}
\big( \ind_{[\frac 1{N^{\ta}}, \frac 2{N^{\ta}}]}^{\ast(k-1)} 
\ast \ind_{[N, N+\frac 1{N^{\ta}}]} \big) (\xi)
\end{aligned}
\label{I11-3}
\end{equation}
for $0<t \ll 1$.
From \eqref{Idec1}--\eqref{I11-3},
we obtain
\eqref{ordukaa}
for $0< t \ll 1$.

Second, we show that
\begin{equation}
|\II(t,\xi)|
\les
N^{-(\ta+1)\be}
\cdot
\eps^k t^{k-1}
N^{-s+(\frac \ta2+\be)k -\be}
\big( \ind_{[\frac 1{N^{\ta}}, \frac 2{N^{\ta}}]}^{\ast(k-1)} 
\ast \ind_{[N, N+\frac 1{N^{\ta}}]} \big) (\xi)
\label{II1}
\end{equation}
for $0< t \ll 1$.
With \eqref{Ftukaa},
Proposition \ref{prop:u1k} and the induction hypothesis yield that
\begin{align*}
&|\II(t,\xi)|
\\
&\les
\sum_{l=2}^{k-1}
\int_0^t
\int_\R
\eps ( \pi^2 2^\be \eps t')^{l-1}
((l-1)!)^{\max (0,\be-1)}
N^{\ta \be + \ta (-\be+\frac 12)l}
\ind_{[\frac 1{N^{\ta}}, \frac {2}{N^{\ta}}]}^{\ast (l)} (\xi-\eta)
\\
&\quad
\times
\eta^\be
\eps^{k-l} t'^{k-l-1}
N^{- s+ (\frac \ta2 + \be) (k-l)- \be}
\big( \ind_{[\frac 1{N^{\ta}}, \frac 2{N^{\ta}}]}^{\ast(k-l-1)} 
\ast \ind_{[N, N+\frac 1{N^{\ta}}]} \big) (\eta)
d\eta dt'
\\
&\les
N^{-(\ta+1)\be}
\cdot
\eps^k t^{k-1}
N^{- s+ (\frac \ta2 + \be) k -\be}
\big( \ind_{[\frac 1{N^{\ta}}, \frac 2{N^{\ta}}]}^{\ast(k-1)} 
\ast \ind_{[N, N+\frac 1{N^{\ta}}]} \big) (\xi).
\end{align*}
Hence,
we have
\eqref{II1} for $0<t \ll 1$.

Finally, we prove that
\begin{equation}
|\III(t,\xi)| 
\les
N^{-(\ta+1)\be}
\cdot
\eps^k t^{k-1}
N^{-s+(\frac \ta2+\be)k -\be}
\big( \ind_{[\frac 1{N^{\ta}}, \frac 2{N^{\ta}}]}^{\ast(k-1)} 
\ast \ind_{[N, N+\frac 1{N^{\ta}}]} \big) (\xi)
\label{III1}
\end{equation}
for $0 <t \ll 1$.
Proposition \ref{prop:u1k} and the induction hypothesis imply that
\begin{equation*}
\begin{aligned}
|\III(t,\xi)| 
&\les
\sum_{l=1}^{k-1} 
\int_0^t
\int_\R
\eps^{l} t'^{l-1}
N^{- s+ (\frac \ta2 + \be) l- \be}
\big( \ind_{[\frac 1{N^{\ta}}, \frac 2{N^{\ta}}]}^{\ast(l-1)} 
\ast \ind_{[N, N+\frac 1{N^{\ta}}]} \big) (\xi-\eta)
\\
&\hspace*{40pt}
\times
\eta^\be
\eps ( \pi^2 2^\be \eps t')^{k-l-1}
((k-l-1)!)^{\max (0,\be-1)}
\\
&\hspace*{60pt}
\times
N^{\ta \be + \ta (-\be+\frac 12)(k-l)}
\ind_{[\frac 1{N^{\ta}}, \frac {2}{N^{\ta}}]}^{\ast (k-l)}
(\eta) d \eta dt'
\\
&\les
N^{-(\ta+1)\be}
\cdot
\eps^k t^{k-1}
N^{-s+(\frac \ta2+\be)k -\be} 
\big( \ind_{[\frac 1{N^{\ta}}, \frac 2{N^{\ta}}]}^{\ast(k-1)} 
\ast \ind_{[N, N+\frac 1{N^{\ta}}]} \big) (\xi).
\end{aligned}
\end{equation*}
Therefore, we obtain \eqref{III1} for $0 <t \ll 1$.

From \eqref{Ftukaa}, \eqref{ordukaa}, \eqref{II1}, and \eqref{III1},
we obtain \eqref{orduk}.
\end{proof}

Using Propositions \ref{prop:u1k} and \ref{prop:Ftuk}, we prove Theorem \ref{thm:ill}.

\begin{proof}[Proof of Theorem \ref{thm:ill}]
We define $k \in \N$ as
\begin{equation}
k :=
\Ceil{
\frac{|\s-s|+1}{\be-\frac \ta 2}}+1.
\label{ch:k}
\end{equation}
It follows from \eqref{d:ta1} that $k$ is well-defined and $k \ge 2$.
Set
\begin{equation}
T:= \frac 1{\log N}.
\label{time}
\end{equation}
From \eqref{inter},
we have
\begin{equation}
\begin{aligned}
\| u(T) \|_{H^\s}
&\ge
\big\| \jb{\cdot}^\s \ind_{[N + \frac {k-1}{N^{\ta}}, N + \frac k{N^{\ta}})} \ft u(T) \big\|_{L^2}
\\
&\ge
\Big\| \sum_{l = \floor{\frac k2}}^k \jb{\cdot}^\s \ind_{[N + \frac {k-1}{N^{\ta}}, N + \frac k{N^{\ta}})} \Ft [u^{(l)}] (T) \Big\|_{L^2}
\\
&\quad
-
\sum_{l = \floor{\frac{k-1+N^{\ta+1}}2}}^{\ceil{k+N^{\ta + 1}}}
\big\| \jb{\cdot}^\s  \ind_{[N + \frac {k-1}{N^{\ta}}, N + \frac k{N^{\ta}})} \Ft[u_1^{(l)}](T) \big\|_{L^2}.
\end{aligned}
\label{uTbda}
\end{equation}
It follows from
Proposition \ref{prop:Ftuk} with \eqref{time}
that
\begin{align*}
&\Big| \sum_{l = \floor{\frac k2}}^k \Ft[u^{(l)}](T,\xi) \Big|
\\
&\ge
\frac 12
\sum_{l = \floor{\frac k2}}^k
\frac{\eps^l T^{l-1}}{(l-1)!} N^{-s+ (\frac \ta 2+\be)l-\be}
\big( \ind_{[\frac 1{N^{\ta}}, \frac 2{N^{\ta}}]}^{\ast(l-1)}
\ast \ind_{[N, N+\frac 1{N^{\ta}}]} \big) (\xi)
\\
&\ge
\frac 12
\frac{\eps^k T^{k-1}}{(k-1)!} N^{-s+ (\frac \ta 2+\be)k-\be}
\big( \ind_{[\frac 1{N^{\ta}}, \frac 2{N^{\ta}}]}^{\ast(k-1)}
\ast \ind_{[N, N+\frac 1{N^{\ta}}]} \big) (\xi)
\end{align*}
for $N \gg 1$ and $\xi \in [ N + \frac{k-1}{N^{\ta}}, N + \frac {k}{N^{\ta}})$.
Hence, we obtain that
\begin{equation}
\begin{aligned}
&\Big\| \sum_{l = \floor{\frac k2}}^k \jb{\cdot}^\s \ind_{[N + \frac {k-1}{N^{\ta}}, N + \frac k{N^{\ta}})} \Ft [u^{(l)}] (T) \Big\|_{L^2}
\\
&\ges
\eps^k T^{k-1}
N^{-s+ (\frac \ta 2+\be)k-\be}
N^{\s - (k-1) \ta - \frac \ta 2}
\\
&=
\eps^k (\log N)^{-(k-1)}
N^{\s-s+ (-\frac \ta 2+\be) (k-1)}
. 
\end{aligned}
\label{ukHs}
\end{equation}
Moreover,
from Proposition \ref{prop:u1k}, Stirling's approximation, \eqref{time}, and \eqref{d:ta1},
we have
\begin{equation}
\begin{aligned}
&
\sum_{l = \floor{\frac{k-1+N^{\ta+1}}2}}^{\ceil{k+N^{\ta + 1}}}
\big\| \jb{\cdot}^\s  \ind_{[N + \frac {k-1}{N^{\ta}}, N + \frac k{N^{\ta}})} \Ft[u_1^{(l)}](T) \big\|_{L^2}
\\
&\les
\sum_{l = \floor{\frac{k-1+N^{\ta+1}}2}}^{\ceil{k+N^{\ta + 1}}}
\eps
(\pi^2 2^\be \eps T )^{l-1}
\big( (l-1)^{l-\frac 12} e^{-l+2} \big)^{\max (0, \be-1)}
\\
&\hspace*{70pt}
\times
N^{\ta \be + \ta ( -\be + \frac 12) l}
N^{-(l-1)\ta}
N^{\s - \frac \ta 2}
\\
&\les
\eps
\end{aligned}
\label{ukHs2}
\end{equation}
for $N \gg 1$.

It follows from \eqref{ch:k} that, for any $0<\eps \ll 1$,
there exists $N \in \N$ with $N \gg 1$ and
\begin{align*}
\frac 1{\log N} < \frac \eps 2,
&
&(\log N)^{-(k-1)} N^{\s - s + (- \frac \ta 2 + \be) (k-1)} \ges \eps^{-k-1}.
\end{align*}
Then, from \eqref{uTbda}--\eqref{ukHs2},
we have
\[
\| u(T) \|_{H^\s}
\ges
\eps^k (\log N)^{-(k-1)}
N^{\s-s+ (-\frac \ta 2+\be)(k-1)}
- \eps
\ge
\eps^{-1}.
\]
With \eqref{phiHs} and \eqref{time}, we obtain Theorem \ref{thm:ill}.
\end{proof}

\section{Ill-posedness on the torus}
\label{sec:illT}

In this section, we prove Theorem \ref{thm:illT}.
We begin by establishing the torus analogue of Lemma \ref{lem:supp}.

\begin{lemma}
\label{lem:suppT}
Let $s \in \R$ and define
\begin{equation}
\wt \nu_0 (t)
:=
\bigg(
\jb{\floor{t}+1}^{2|s|} (t-\floor{t})
+
\sum_{n=1}^{\floor{t}} \jb{n}^{2|s|}
\bigg)^{\frac 12}
\label{nu00}
\end{equation}
for $t \ge 0$,
where $\sum_{n=1}^{\floor{t}} \jb{n}^{2|s|} :=0$ if $t \in [0,1)$.
Suppose that $F \in H^s(\T)$ satisfies $\supp \ft F \subset \N$. 
Then, we have $\ft F \in \X^{0,\wt \nu_0}$.
In particular, $F \in \Ft^{-1} \X$. 
\end{lemma}

\begin{proof}
Note that $\wt \nu_0$ defined in \eqref{nu00} belongs to $C_0^\up$.
Let $l > 0$ and $\psi \in \D_{< l}(\R)$.
From $F \in H^s(\T)$, we have $\ft F \in \l^2$.
H\"older's inequality, $\supp \ft F \subset \N$, and \eqref{nu00} imply that
\begin{align*}
|\dual {\ft F}\psi |
&\le
\sum_{n=1}^\infty
|\ft F(n)| |\psi(n)|
\\
&\le
\Big( \sum_{n=1}^{\floor{l}} \jb{n}^{-s} \jb{n}^s |\ft F(n)| \Big) \|\psi\|_0^+ 
\\
&\le
\Big( \sum_{n=1}^{\floor{l}} \jb{n}^{2|s|} \Big)^{\frac 12}
\|F\|_{H^s} \|\psi\|_0^+
\\
&\le
\wt \nu(l)
\|F\|_{H^s} \|\psi\|_0^+. 
\end{align*}
With \eqref{norm}, we obtain that
\[
\rho_l^{0, \wt \nu_0}(\ft F) 
\le \|F\|_{H^s}
\]
for $l >0$.
Hence, we have
\[
\rho^{0, \wt \nu_0}(\ft F)
= 
\sup_{l > 0} \rho_l^{0, \wt \nu_0}(\ft F) 
\le 
\|F \|_{H^s} 
< \infty.
\]
Then, $\ft F \in \X^{0, \wt \nu_0}$.
\end{proof}

\begin{remark}
\rm
We can choose $\wt \nu_0$ as
\[
\wt \nu_0 (t)
=
\bigg(
\jb{\floor{t}+1}^{-2s} (t-\floor{t})
+
\sum_{n=1}^{\floor{t}} \jb{n}^{-2s}
\bigg)^{\frac 12},
\]
that is,
we can replace ``$|s|$'' in \eqref{nu00} with ``$-s$''.
In this case,
we have
$\sup_{t \ge 0} \wt \nu_0 (t) \le 1$ for large $s$.
However,
this would require us to replace $\wt \nu_0$ in the proof of Theorem \ref{thm:wel-T}.
For this reason,
we choose $\wt \nu_0$ as in \eqref{nu00}.
\end{remark}

\begin{remark}
\label{rem:Xn1}
\rm
In the assumptions of Lemma \ref{lem:suppT},
we can not replace 
\[
\supp \ft F \subset \N
\]
with
\[
\supp \ft F \subset \N_0.
\]
More precisely,
if $F \in C^\infty (\T)$ with
$\supp \ft F \subset \N_0$ and $\ft F(0) \neq 0$,
we have
\[
F \notin \Ft^{-1} \X.
\]
Indeed,
we have
\[
| \dual{\ft F}{X^l 1} | 
= |\ft F(0)| |(X^l1)(0)|
= |\ft F(0)| \neq 0
\]
for $l \in (0, \frac 12)$,
where $X^l$ is defined in \eqref{X^_l}.
Moreover,
when $\mu \in C_0^\up$, we have
\[
| \dual{\ft F}{X^l 1} | 
= |\ft F(0)| \| X^l 1 \|_{\mu (l)}^+
\]
for $0<l \ll 1$.
It follows from \eqref{norm} that
\[
\rho^{\mu,\nu}(\ft F)
\ge
\rho_l^{\mu, \nu} (\ft F)
\ge \nu(l)^{-1} |\ft F(0)|
\]
for $\mu, \nu \in C_0^\up$ and $0< l \ll 1$.
Since $\nu \in C_0^\up$, in particular, $\nu(0) =0$ and $\nu$ is continuous,
we get $\rho^{\mu,\nu}(\ft F)= \infty$.
We thus obtain $F \notin \Ft^{-1} \X$.
\end{remark}

For the Cauchy problem \eqref{dNLS+}, we obtain the following theorem.
Although this essentially repeats the argument of \cite{NaWa25},
we include a proof for the sake of completeness.

\begin{theorem}
\label{thm:wel-T}
Let $\al, \be > 0$.
Let $\wt \nu_0$ be as in \eqref{nu00}.
For any $b_0>0$,
there exists
\[
\wt \nu \in C([0,\infty); C_0^\up)
\]
such that the followings hold:

\begin{enumerate}
\item
$\wt \nu (t,l) \ge \wt \nu_0(l)$ for all $t,l \ge 0$.

\item
For any $\wt \phi \in \Ft^{-1} \X^{0, \wt \nu_0}$
with $\rho^{0, \wt \nu_0}(\Ft [\wt \phi]) \le b_0$,
there exists a unique solution $v \in C([0,1]; \Ft^{-1} \X)$ to \eqref{dNLS+}
satisfying
\[
\sup_{0 \le t \le 1}
\rho^{0, \wt \nu (t)}( \ft v(t) ) \le 2 b_0.
\]
\end{enumerate}
\end{theorem}

\begin{proof}
We solve the corresponding integral equation to \eqref{dNLS+} as the fixed point of
the operator
\begin{equation}
\begin{aligned}
\wt \G(\Ft [\wt \phi], \ft v) (t,n)
&:= e^{-it |n|^\al + t \ft \phi(0) |n|^\be} \Ft [\wt \phi ] (n)
\\
&\qquad
+ \int_0^t e^{-i(t-t') |n|^\al + (t-t')\ft \phi (0) |n|^\be} \Ft[v D^\be v](t', n) dt'.
\end{aligned}
\label{int-T}
\end{equation}

Set
\begin{equation}
l_\ast := \sup \wt \nu_0^{-1}([0,1]).
\label{eq:last}
\end{equation}
Define
\begin{equation}
L(l)
:=
\jb{\Re \ft \phi(0)}
\max (l_\ast, l)^\be
\label{eq:Ll1}
\end{equation}
for $l \ge 0$.
For any $t \in [0,1]$, $l > 0$, and $f \in L^\infty((-\infty,l))$, we have
\begin{equation}
\| e^{-it |\cdot|^\al + t \ft \phi(0) |\cdot|^\be} f \|_{0}^+
\le
e^{ t |\Re \ft \phi(0)| l^\be } \|f \|_0^+
\le
e^{ t L(l)} \|f \|_0^+.
\label{LNhat-T}
\end{equation}

We also fix $\eps > 0$ small enough such that
\begin{equation}
\eps < \frac 1{8 b_0}.
\label{esp2}
\end{equation}
Define $\nu_\ast, \nu_1, \nu : [0, \infty) \to C_0^\up$ and $\kappa \in C^\down_{(0,1]}$ as 
\begin{equation}
\begin{aligned}
&\wt \nu_\ast (t,l) := \eps^{-1} \wt \nu_0 (l)
e^{t L(l)},
&&
\kappa (l) := \frac 1{1 + l^\be},
\\
&\wt \nu_1 (t,l) := \CC (\wt \nu_\ast (t), \kappa) (l),
&&
\wt \nu(t,l) := \eps \kappa (l) \wt \nu_1(t,l), 
\end{aligned}
\label{nus}
\end{equation}
for $t, l \ge 0$,
where $\wt \nu_\ast (t) := \wt \nu_\ast (t, \cdot)$ and $\CC$ is defined in Lemma \ref{lem:convo}. 

Let
$\wt \phi \in \Ft^{-1} \X^{0, \wt \nu_0}$
satisfy $\rho^{0, \wt \nu_0}(\Ft [\wt \phi]) \le b_0$.
We will prove that
$\wt \G(\Ft [\wt \phi], \ft v)$ is a contraction mapping on
\begin{equation}
\wt Z := \Big\{ \ft v \in C ([0,1]; \X) \mid \sup_{0 \le t \le 1} \rho^{0, \wt \nu (t)} (\ft v(t)) \le 2 b_0  \Big\}.
\label{sp-T}
\end{equation}
For any $\ft v \in \wt Z$, and $\psi \in \D_{<l}$,
it follows from \eqref{int-T}, \eqref{norm}, \eqref{LNhat-T}, and Lemma \ref{lem:bd1} that
\begin{equation*}
\begin{aligned}
&\Re \dual{\wt \G(\Ft [\wt \phi], \ft v) (t)}{\psi} 
\\
&\le
| \dual{e^{-it|\cdot|^\al + t \ft \phi (0) |\cdot|^\be} \Ft [\wt \phi]}{\psi} |
\\
&\quad
+
\Big|\Dual {\int_0^t e^{-i(t-t')|\cdot|^\al + (t-t') \ft \phi (0) |\cdot|^\be} \Ft[v D^\be v](t') dt'}{\psi}\Big| 
\\
&\le
\wt \nu_0 (l) e^{t L(l)} \rho_l^{0,\wt \nu_0} (\Ft [\wt \phi]) \| \psi \|_0^+
\\
&\quad
+
\int_0^t
\kappa (l) \wt \nu_1 (t',l) e^{ (t-t') L(l)} \rho_l^{0,\kappa \wt \nu_1 (t')}(\Ft [v D^\be v](t'))
\| \psi \|_0^+
dt'
.
\end{aligned}
\end{equation*}

With \eqref{rholm1} and $\rho^{0, \wt \nu_0}(\Ft [\wt \phi]) \le b_0$,
we obtain
\begin{equation}
\begin{aligned}
&\rho_l^0 (\G(\Ft [\wt \phi], \ft u) (t))
\\
&\le
\wt \nu_0 (l) e^{t L(l)} b_0
+
\int_0^t
\kappa(l) \wt \nu_1(t',l)
e^{ (t-t') L(l)}
\rho_l^{0, \kappa \wt \nu_1 (t')}(\Ft [v D^\be v](t')) dt'.
\end{aligned}
\label{estDuht}
\end{equation}

The same calculation as in \eqref{estcon} with Lemma \ref{lem:convo} and \eqref{nus}
yields that
\begin{equation}
\rho_l^{0, \kappa \wt \nu_1(t)}(\Ft [v D^\be v](t))
\le
\eps^2 \big(\rho_l^{0, \wt \nu (t)} (\ft v(t)) \big)^2
\label{estcont}
\end{equation}
for $t,l>0$.
From \eqref{estDuht}, \eqref{estcont}, and \eqref{nus}, 
we have
\begin{equation}
\begin{aligned}
\rho_l^0 (\wt \G(\Ft [\wt \phi], \ft v) (t))
&\le
\wt \nu_0 (l) e^{t L(l)} b_0
\\
&\quad
+
\int_0^t
\kappa(l) \wt \nu_1(t',l) e^{(t-t') L(l)}
\eps^2
\big(\rho_l^{0, \wt \nu (t')} (\ft v(t')) \big)^2 dt'
\\
&\le
\wt \nu_0 (l) e^{t L(l)} b_0
\\
&\quad
+
\int_0^t
\wt \nu (t',l) e^{(t-t') L(l)}
\eps
\big(\rho_l^{0, \wt \nu (t')} (\ft v(t')) \big)^2 dt'
\end{aligned}
\label{estDuh1}
\end{equation}
for $\ft v \in \wt Z$.

Lemma \ref{lem:convo} (i) and \eqref{nus} yield that
\[
\wt \nu (t,l)
\ge \eps \kappa (l) \wt \nu_1 (t,l)
\ge \eps \wt \nu_\ast  (t,l)
= \wt \nu_0 (l) e^{t L(l)}.
\]
With \eqref{estDuh1} and $\ft v \in \wt Z$,
we have
\[
\rho^{0, \wt \nu (t)}(\wt \G(\Ft [\wt \phi], \ft v) (t))
\le
b_0
+
4\eps b_0^2
\sup_{l>0}
\int_0^t e^{(t-t') L(l)}
\frac{\wt \nu(t',l)}{\wt \nu (t,l)} dt'
\]
for $0<t<1$.
Here,
from \eqref{nus},
\[
\wt \nu_\ast (t',l) \le \kappa (l)
\]
is equivalent to
\[
\wt \nu_0(l) e^{t' L(l)}
\le \frac{\eps}{1+l^\be}.
\]
By \eqref{eq:last} and \eqref{eq:Ll1},
this inequality holds only for $l \in (0,l_\ast)$.
Then,
Lemma \ref{lem:convo} (iii) and \eqref{nus} implies that
\begin{align*}
e^{(t-t') L(l)}
\frac{\wt \nu(t',l)}{\wt \nu (t,l)}
&=
e^{(t-t') L(l)}
\frac{\CC (\wt \nu_\ast (t'), \kappa) (l)}{\CC (\wt \nu_\ast (t), \kappa) (l)}
\\
&\le
\frac{\CC (e^{(t-t') L} \wt \nu_\ast (t'), \kappa) (l)}{\CC (\wt \nu_\ast (t), \kappa) (l)}
=1
\end{align*}
for $0 \le t' \le t$.
Hence, we obtain that
\[
\sup_{0 \le t \le 1}
\rho^{0, \wt \nu (t)}(\wt \G(\Ft [\wt \phi], \ft v) (t))
\le
b_0
+
4\eps b_0^2.
\]
It follows from \eqref{esp2}
that
$\wt \G(\Ft [\wt \phi], \cdot)$ is a mapping on $\wt Z$.

A similar calculation yields that
$\wt \G(\Ft [\wt \phi], \cdot)$ is a contraction mapping on $\wt Z$.
Therefore, we have a unique solution $\ft v \in \wt Z$ by using fixed point argument.
\end{proof}

By using Theorem \ref{thm:wel-T}, we prove Theorem \ref{thm:illT}.
\begin{proof}[Proof of Theorem \ref{thm:illT}]

For $s \in \R$ and $N \in \N$ with $N \ge 3$, we take the initial data as
\begin{equation*}
\phi(x) := \frac {1 + \jb{N}^{-s} e^{iNx}}{\log N}.
\end{equation*}
Then, we have 
\begin{equation}
\| \phi \|_{H^s} \le \frac 2{\log N}.
\label{Hsini}
\end{equation}

Define
\begin{equation}
\wt \phi (x) :=
\phi(x) - \ft \phi (0)
= \frac{\jb{N}^{-s} e^{iNx}}{\log N}.
\label{eq:wtphi}
\end{equation}
Lemma \ref{lem:suppT} yields that
$\Ft [\wt \phi] \in \X^{0, \wt \nu_0}$,
where $\wt \nu_0$ is defined in \eqref{nu00}.
By Theorem \ref{thm:wel-T},
there exists a solution
$v \in C([0,1]; \Ft^{-1} \X)$ to \eqref{dNLS+}.
Moreover,
$v$ satisfies
\[
\ft v(t,n)
= 
\wt \G(\Ft [\wt \phi], \ft v) (t,n)
\]
in $\wt Z$ defined in \eqref{sp-T},
where
$\wt \G(\Ft [\wt \phi], \ft v)$ is defined in \eqref{int-T}.
Hence,
\begin{equation}
\begin{aligned}
&\dt \ft v(t,n)
+ i |n|^\al \ft v(t,n)
\\
&\quad
=
\ft \phi(0) |n|^\be \ft v(t,n)
+
\sum_{\substack{n_1,n_2 \in \N \\ n_1+n_2=n}} \ft v(t,n_1) |n_2|^\be \ft v(t,n_2)
\end{aligned}
\label{eq:vv1}
\end{equation}
holds for $t \in [0,1]$ and $n \in \N$.

We set
\begin{align*}
v^{(1)}(t)
&:= e^{-it D^\al + t \ft \phi (0) D^\be} \wt \phi,
\\
v^{(k)}(t)
&:= \sum_{\substack{k_1,k_2 \in \N \\ k_1+k_2=k}}
\int_0^t e^{-i (t-t') D^\al + (t-t') \ft \phi (0) D^\be}
(v^{(k_1)} D^\be v^{(k_2)}) (t') dt'
\end{align*}
for $k \ge 2$.
By
(the proof of) Theorem \ref{thm:wel-T},
we have
\[
v = \sum_{k=1}^\infty v^{(k)}
\]
in $\Ft^{-1} \wt Z$,
where $\wt Z$ is defined in \eqref{sp-T}.
It follows from \eqref{eq:wtphi} that
\[
\supp \Ft [v^{(k)}] (t, \cdot) \subset \{kN\}
\]
for $t \in [0,1]$ and $k \in \N$.
Hence, we have
\begin{equation}
|\ft v (t,N) | 
= |\ft v^{(1)} (t,N) | 
= e^{t (\Re \ft \phi(0)) N^\be} \frac {\jb{N}^{-s}}{\log N}.
\label{ftvN}
\end{equation}

Set
\[
u := v + \ft \phi(0).
\]
With \eqref{eq:vv1},
$u$ is a solution in the sense of Theorem \ref{thm:illT} (ii).
Let $\s \in \R$.
We set
\begin{equation}
T:= (|\s-s| + 1) \frac{(\log N)^2}{N^{\be}}.
\label{time:T}
\end{equation}
From \eqref{ftvN} and \eqref{time:T}, we obtain
\begin{equation}
\begin{aligned}
\| u(T) \|_{H^\s} 
& 
\ge
\jb{N}^\s | \ft v(T, N) | 
=
\jb{N}^\s e^{T (\Re \ft \phi(0)) N^\be} \frac {\jb{N}^{-s}}{\log N}
\\
&=
\frac{\jb{N}^{\s-s}}{\log N} N^{|\s-s|+1}
\ges
\frac {N}{\log N}.
\end{aligned}
\label{ulow}
\end{equation}

For any $\be > 0$ and $\eps > 0$, there exists $N \in \N$ with $N \gg 1$ and
\[
\frac 2 {\log N} < \eps, 
\quad (|\s-s| + 1) \frac{(\log N)^2}{N^{\be}} < \eps,
\quad \frac N{\log N} > \eps^{-1}.
\]
By \eqref{Hsini}, \eqref{time:T}, and \eqref{ulow}, we have Theorem \ref{thm:illT}.
\end{proof}

\appendix

\section{Well-posedness on the real-line}
\label{app:WP}

In this appendix, on the real-line, we use the smoothing effect due to dispersion to show that when $\be \le \frac{\al-1}2$,
the Cauchy problem \eqref{dNLS} is well-posed in a Sobolev space.

\begin{theorem}
\label{thm:WP}
Let $\al, \be , s \in \R$ satisfy
\begin{equation}
\al >1, \quad
0\le \be \le \frac{\al-1}2, \quad
s>
\max \Big(
\be + \frac 12,
\frac \al 4
\Big)
.
\label{eq:WP1}
\end{equation}
Then,
the Cauchy problem
\eqref{dNLS} is well-posed in $H^s(\R)$.
\end{theorem}

The third condition in \eqref{eq:WP1} is probably not optimal, but we do not pursue optimal regularity here.

We use the following notation:
for $p,q \in [1,\infty]$ and $T>0$,
\[
\| f \|_{L_T^p L_x^q} := \| f \|_{L^p([0,T];  L^q(\R))},
\quad
\| f \|_{L_x^q L_T^p} := \| f \|_{L^q(\R;  L^p([0,T]))}.
\]
First, we recall the Strichartz estimate;
see Theorem 2.1 in \cite{KPV91}.

\begin{theorem}
\label{thm:Str}
Let $\al>1$.
Then, we have
\[
\| D^{\frac{\al-2}4} e^{-it D^\al} \phi \|_{L_t^4 L_x^\infty}
\les \| \phi \|_{L^2}
\]
for any $\phi \in L^2(\R)$.
\end{theorem}

We also use the local smoothing estimate;
see Theorem 4.1 in \cite{KPV91}.

\begin{theorem}
\label{thm:locs}
Let $\al>1$.
Then, we have
\[
\| D^{\frac{\al-1}2} e^{-it D^\al} \phi \|_{L_x^\infty L_t^2}
\les \| \phi \|_{L^2}
\]
for any $\phi \in L^2(\R)$.
\end{theorem}

In addition, we use the maximal function estimate;
see Corollary 2.9 in \cite{KPV91b}.

\begin{theorem}
\label{thm:KVP2}
Let $\al>1$, $s>\frac \al4$, and $T \in (0,1)$.
Then,
we have
\[
\| e^{-it D^\al} \phi \|_{L_x^2 L_T^\infty}
\les \| \phi \|_{H^s}
\]
for any $\phi \in H^s(\R)$.
\end{theorem}

Although Corollary 2.9 in \cite{KPV91b} is proved for $\al \ge 2$,
the same argument applies for $\al>1$.

We also use the higher-order fractional Leibniz rule;
see Theorem 1.2 (2) in \cite{Li19}.
We denote by $D^{s,k}$ the Fourier multiplier with the corresponding symbol
\[
(-i \pd_\xi)^k (|\xi|^s)
\]
for $s>0$ and $k \in \N_0$.

\begin{theorem}
\label{thm:fracLeib}
Let $s>0$.
For any $s_1, s_2 \ge 0$ with $s_1+s_2=s$
and any Schwartz functions $f,g$,
we have
\begin{align*}
&\bigg\|
D^s (fg)
- \sum_{\substack{j \in \N_0 \\ j < s_1}} \frac 1{j!} (\dx^j f) (D^{s,j} g)
- \sum_{\substack{k \in \N_0 \\ k \le s_2}} \frac 1{k!} (D^{s,k} f) (\dx^k g)
\bigg\|_{L^2}
\\
&\les
\| D^{s_1} f \|_{L^2} \| D^{s_2} g \|_{L^\infty}.
\end{align*}
\end{theorem}

\begin{proof}[Proof of Theorem \ref{thm:WP}]
We consider the corresponding integral equation to \eqref{dNLS}:
\[
u(t)
= e^{-it D^\al} \phi
+ \int_0^t e^{-i (t-t') D^\al} (u D^\be u) (t') dt'.
\]
We will prove that
\begin{equation}
\G [u] (t)
:=
e^{-it D^\al} \phi
+ \int_0^t e^{-i (t-t') D^\al} (u D^\be u) (t') dt'
\label{eq:Ga1}
\end{equation}
is a contraction mapping on an appropriate space.
Since the $\be =0$ is straightforward,
we only consider the case $\be >0$.
We divide the proof into two cases:
$\be > \max (\frac{\al-2}4, 0)$ and $0 <\be \le \frac{\al-2}4$.

(i) The case $\be > \max (\frac{\al-2}4, 0)$:
It follows from \eqref{eq:WP1} that
\[
s-\be+\frac{\al-2}4
\in (0,s).
\]
For $T \in (0,1)$,
define
\begin{equation}
\begin{aligned}
\| u \|_{X_T^s}
&:=
\| u \|_{L_T^\infty H^s_x}
+ \| D^{s+\frac{\al-2}4} u \|_{L_T^4 L_x^\infty}
\\
&\quad
+
\sum_{\substack{j \in \N_0 \\ j<\be-\frac{\al-2}4}}
\big( \| \dx^j u \|_{L_x^2 L_T^\infty}
+
\| D^{s,j} D^\be u \|_{L_x^\infty L_T^2} \big)
\\
&\quad
+
\sum_{\substack{k \in \N_0 \\ k \le s-\be+\frac{\al-2}4}}
\big(
\| D^{s,k} u \|_{L_T^\infty L_x^2}
+
\| \dx^k D^\be u \|_{L_T^4 L_x^\infty}
\big).
\end{aligned}
\label{XsT1}
\end{equation}

When $0 \le j<\be-\frac{\al-2}4$,
it follows from \eqref{eq:WP1} that
\begin{align*}
&j+ \frac \al 4
< \be + \frac 12
<s,
\\
&
s-j+\be - \frac{\al-1}2
> s - \frac \al 4 >0,
\\
&-j+ \be - \frac{\al-1}2
\le
\be - \frac{\al-1}2
\le 0.
\end{align*}
In addition,
when $0 \le k \le s-\be+\frac{\al-2}4$,
it follows from \eqref{eq:WP1} that
\begin{align*}
&
0<
\be - \frac{\al-2}4
\le
k+ \be - \frac{\al-2}4
\le
s.
\end{align*}
Hence,
Theorems \ref{thm:Str}, \ref{thm:locs}, and \ref{thm:KVP2} with \eqref{eq:Ga1} and \eqref{XsT1} yield that
\begin{align*}
\| \G [u] \|_{X_T^s}
&\les
\| \phi \|_{H^s}
+
\| u D^\be u \|_{L_T^1 L_x^2}
+
\| D^s (u D^\be u) \|_{L_T^1 L_x^2}
\\
&\les
\| \phi \|_{H^s}
+
T^{\frac 12}
\| u D^\be u \|_{L_{T,x}^2}
+
T^{\frac 12}
\| D^s (u D^\be u) \|_{L_{T,x}^2}.
\end{align*}

Sobolev's embedding yields that
$H^{s_0}(\R) \hookrightarrow L^\infty(\R)$
for $s_0>\frac 12$.
By choosing $s_0 \in (\frac 12 ,s-\be)$,
we have
\[
\| D^{\be} u \|_{L_{T,x}^\infty}
\les
\| D^{\be} u \|_{L_T^\infty H_x^{s_0}}
\les
\| u \|_{L_T^\infty H_x^s}.
\]
Then,
H\"older's inequality with \eqref{XsT1} implies that
\[
\| u D^\be u \|_{L_{T,x}^2}
\le
T^{\frac 12}
\| u \|_{L_T^\infty L_x^2}
\| D^\be u \|_{L_{T,x}^\infty}
\les
T^{\frac 12}
\| u \|_{X_T^s}^2.
\]
Theorem \ref{thm:fracLeib} with \eqref{XsT1} yields that
\begin{align*}
&\bigg\| D^s (u D^\be u)
-
\sum_{\substack{j \in \N_0 \\ j<\be-\frac{\al-2}4}} \frac 1{j!} (\dx^j u) (D^{s,j} D^\be u)
\\
&\hspace*{70pt}
-
\sum_{\substack{k \in \N_0 \\ k \le s-\be+\frac{\al-2}4}} \frac 1{k!} (D^{s,k} u) (\dx^k D^\be u)
\bigg\|_{L_{T,x}^2}
\\
&\les
T^{\frac 14}
\| D^{\be-\frac{\al-2}4} u \|_{L_T^\infty L_x^2}
\| D^{s+\frac{\al-2}4} u \|_{L_T^4L_x^\infty}
\\
&\les
T^{\frac 14}
\| u \|_{X_T^s}^2
.
\end{align*}
For $j <\be-\frac{\al-2}4$,
H\"older's inequality with \eqref{XsT1} yields that
\[
\| (\dx^j u) (D^{s,j} D^\be u) \|_{L_{T,x}^2}
\le
\| \dx^j u \|_{L_x^2 L_T^\infty} \| D^{s,j} D^\be u \|_{L_x^\infty L_T^2}
\le
\| u \|_{X_T^s}^2.
\]
For $k \le s-\be+\frac{\al-2}4$,
H\"older's inequality with \eqref{XsT1} yields that
\[
\| (D^{s,k} u) (\dx^k D^\be u) \|_{L_{T,x}^2}
\le
T^{\frac 14}
\| D^{s,k} u \|_{L_T^\infty L_x^2} \| \dx^k D^\be u \|_{L_T^4 L_x^\infty}
\le
T^{\frac 14}
\| u \|_{X_T^s}^2.
\]
Combining estimates above,
we obtain that
\[
\| \G [u] \|_{X_T^s}
\le
C_0
\| \phi \|_{H^s}
+
C_1
T^{\frac 12}
\| u \|_{X_T^s}^2,
\]
where $C_0, C_1 >0$ are constants.

A similar calculation yields
the difference estimate:
\[
\| \G [u] - \G [v] \|_{X_T^s}
\le
C_2
T^{\frac 12}
( \| u \|_{X_T^s} + \| v \|_{X_T^s})
\| u-v \|_{X_T^s}
.
\]
Hence,
by the contraction mapping theorem,
there exists a solution to
\[
\G[u]=u
\]
in
\[
\{ u \in C([0,T]; H^s(\R)) \mid
\| u \|_{X_T^s} \le 2 C_0 \| \phi \|_{H^s} \}
\]
provided that we choose $0< T \ll \min (1, \| \phi \|_{H^s}^{-2})$.
The uniqueness and continuous dependence on the initial data then follow from a standard argument, and we omit the details.

(ii) The case $0 < \be \le \frac{\al-2}4$:
It follows from \eqref{eq:WP1} that
\[
s+\be- \frac{\al-2}4
\in (0,s].
\]
For $T \in (0,1)$,
define
\begin{equation}
\begin{aligned}
\| u \|_{\wt X_T^s}
&:=
\| u \|_{L_T^\infty H^s_x}
+
\sum_{\substack{j \in \N_0 \\ j<s+\be-\frac{\al-2}4}}
\big( \| \dx^j u \|_{L_T^\infty L_x^2}
+
\| D^{s,j} D^\be u \|_{L_T^4 L_x^\infty} \big)
.
\end{aligned}
\label{XsT2}
\end{equation}

When $0 \le j<s+\be-\frac{\al-2}4$,
it follows from \eqref{eq:WP1} that
\begin{align*}
&
s-j+\be - \frac{\al-2}4
>0,
\\
&-j+ \be - \frac{\al-2}4
\le
\be - \frac{\al-2}4
\le 0.
\end{align*}
Theorems \ref{thm:Str}, \ref{thm:locs}, and \ref{thm:KVP2} with \eqref{eq:Ga1} and \eqref{XsT2} yield that
\begin{align*}
\| \G [u] \|_{\wt X_T^s}
&\les
\| \phi \|_{H^s}
+
\| u D^\be u \|_{L_T^1 L_x^2}
+
\| D^s (u D^\be u) \|_{L_T^1 L_x^2}
\\
&\les
\| \phi \|_{H^s}
+
T^{\frac 12}
\| u D^\be u \|_{L_{T,x}^2}
+
T^{\frac 12}
\| D^s (u D^\be u) \|_{L_{T,x}^2}.
\end{align*}

Theorem \ref{thm:fracLeib} with \eqref{XsT2} yields that
\begin{align*}
&\bigg\| D^s (u D^\be u)
-
\sum_{\substack{j \in \N_0 \\ j<s+\be-\frac{\al-2}4}} \frac 1{j!} (\dx^j u) (D^{s,j} D^\be u)
\\
&\hspace*{70pt}
-
\sum_{\substack{k \in \N_0 \\ k \le \frac{\al-2}4-\be}} \frac 1{k!} (D^{s,k} u) (\dx^k D^\be u)
\bigg\|_{L_{T,x}^2}
\\
&\les
T^{\frac 12}
\| D^{s+\be-\frac{\al-2}4} u \|_{L_T^\infty L_x^2}
\| D^{\frac{\al-2}4} u \|_{L_{T,x}^\infty}
\\
&\les
T^{\frac 12}
\| u \|_{\wt X_T^s}^2
.
\end{align*}
For $j<s+\be-\frac{\al-2}4$,
H\"older's inequality with \eqref{XsT2} yields that
\[
\| (\dx^j u) (D^{s,j} D^\be u) \|_{L_{T,x}^2}
\le
T^{\frac 14}
\| \dx^j u \|_{L_T^\infty L_x^2} \| D^{s,j} D^\be u \|_{L_T^4 L_x^\infty}
\les
\| u \|_{\wt X_T^s}^2.
\]
For $k \le \frac{\al-2}4-\be$,
H\"older's inequality and Sobolev's embedding with \eqref{XsT2} yields that
\[
\| (D^{s,k} u) (\dx^k D^\be u) \|_{L_{T,x}^2}
\le
T^{\frac 12}
\| D^{s,k} u \|_{L_T^\infty L_x^2} \| \dx^k D^\be u \|_{L_{T,x}^\infty}
\les
T^{\frac 12}
\| u \|_{\wt X_T^s}^2.
\]
Combining estimates above,
we obtain that
\[
\| \G [u] \|_{\wt X_T^s}
\le
\wt C_0
\| \phi \|_{H^s}
+ \wt C_1 T^{\frac 12}
\| u \|_{\wt X_T^s}^2,
\]
where $\wt C_0, \wt C_1>0$ are constants.
Since the remaining arguments are identical to those in case (i),
we omit the details.
\end{proof}

\mbox{}

\noindent
{\bf 
Acknowledgements.}
This work was
supported by JSPS KAKENHI Grant numbers
JP23K03182 and JP26K00613 and JST SPRING Grant Number JPMJSP2132.

\end{document}